\documentclass[12pt,A4]{article}
\usepackage[english]{babel}
\usepackage{amsmath,amsfonts,amssymb,amsthm,mathrsfs,bbm}
\usepackage[font=sf, labelfont={sf,bf}, margin=1cm]{caption}
\usepackage{graphicx,graphics}
\usepackage{epsfig}
\usepackage{latexsym}
\usepackage[applemac]{inputenc}
\usepackage{ae,aecompl}
\usepackage{pstricks}
\usepackage{enumerate}
\usepackage{xcolor}
\usepackage[pdfpagemode=UseNone,bookmarksopen=false,colorlinks=true,urlcolor=blue,citecolor=blue,citebordercolor=blue,linkcolor=blue]{hyperref}

\usepackage[normalem]{ulem}

\usepackage[top=2.5cm,left=3.5cm,right=3.5cm,bottom=2.5cm]{geometry}
\usepackage{comment}

\usepackage[all, knot, cmtip]{xy}

\newcommand{\ndN}{\mathbb{N}}
\newcommand{\ndZ}{\mathbb{Z}}

\newcommand{\ndR}{\mathbb{R}}

\renewcommand{\Pr}[1]{\mathbb{P}(#1)}

\newcommand{\Ex}[1]{\mathbb{E}[#1]}
\newcommand{\Exb}[1]{\mathbb{E}\left[#1\right]}

\newcommand{\one}{\mathbbm{1}}

\newcommand{\cT}{\mathcal{T}}

\newcommand{\cV}{\mathcal{V}}

\newcommand{\cH}{\mathcal{H}}

\newcommand{\cP}{\mathcal{P}}

\newcommand{\cQ}{\mathcal{Q}}

\newcommand{\fmT}{\mathfrak{T}}

\newcommand{\cE}{\mathcal{E}}

\newcommand{\he}{\mathrm{h}}

\newcommand{\spa}{\mathsf{span}}

\newcommand{\UHT}{\mathcal{U}_{\infty}}
\newcommand{\VHT}{\mathcal{V}_{\infty}}

\newcommand{\eqdist}{\,{\buildrel d \over =}\,}

\newcommand{\convdis}{\,{\buildrel d \over \longrightarrow}\,}
\newcommand{\convp}{\,{\buildrel p \over \longrightarrow}\,}

\newtheorem{theorem}{Theorem}[section]

\newtheorem{proposition}[theorem]{Proposition}
\newtheorem{lemma}[theorem]{Lemma}

\numberwithin{equation}{section}

\title{\textbf{Local limits of large Galton--Watson trees rerooted at a random vertex}}
\date{}

\author{Benedikt Stufler\thanks{\'Ecole Normale Sup\'erieure de Lyon, E-mail: benedikt.stufler@ens-lyon.fr; The author is supported by the German Research Foundation DFG, STU 679/1-1}}

\begin{document}
	
	\maketitle
	
\let\thefootnote\relax\footnotetext{ \\\emph{MSC2010 subject classifications}. 60J80, 60B10 }

\vspace {-0.5cm}

\begin{abstract}
We discuss various forms of convergence of the vicinity of a uniformly at random selected vertex in random simply generated trees, as the size  tends to infinity. For the standard case of a critical Galton--Watson tree conditioned to be large the limit is the invariant random  sin-tree constructed by Aldous~(1991). In the condensation regime, we describe in complete generality the asymptotic local behaviour from a random vertex up to its first ancestor with large degree. Beyond this distinguished ancestor, different behaviour may occur, depending on the branching weights. In a  subregime of complete condensation, we obtain convergence toward a novel limit tree, that  describes the asymptotic shape of the vicinity of the full path from a random vertex to the root vertex. This includes the case where the offspring distribution follows a power law up to a factor that varies slowly at infinity. 
\end{abstract}

\section{Introduction}

The study of the asymptotic local behaviour of the vicinity of the fixed root vertex of random trees has received considerable attention in recent literature. Jonsson and Stef\'ansson~\cite{MR2764126} described a phase transition between an infinite spine case and a condensation setting for large Galton--Watson trees with a power-law offspring distribution. A third regime for random simply generated trees with superexponential branching weights was studied by Janson, Jonsson and Stef\'ansson~\cite{MR2860856}. The asymptotic shape of large simply generated trees as their size tends to infinity was later described in complete generality by Janson~\cite{MR2908619}. Abraham and Delmas~\cite{MR3227065, MR3164755} classified the limits of conditioned Galton--Watson trees as the total number of vertices with outdegree in a given fixed set tends to infinity. Limits of Galton--Watson trees having a large number of protected nodes were established by Abraham, Bouaziz, and Delmas~\cite{2015arXiv150902350A}. The asymptotic shape of conditioned multi-type Galton--Watson trees was studied  by Stephenson~\cite{Stephenson2016}, Abraham, Delmas, and Guo \cite{2015arXiv151101721A}, and P{\'e}nisson~\cite{MR3476213}.

Clearly considerable effort and progress is being made in understanding  local limits of random trees that describe the asymptotic behaviour near the fixed root vertex, and for random simply generated trees even a complete classification is available. As for the question of the asymptotic shape of the vicinity of a random vertex, Aldous~\cite{MR1102319} studied in his pioneering work asymptotic fringe distributions for general families of random trees. For the case of critical Galton--Watson trees, he established, at least when the offspring distribution has finite variance, convergence of the tree obtained by rerooting at a random vertex. A recent work by Holmgren and Janson~\cite{2016arXiv160103691H} studied fringe trees and extended fringe trees of models of random trees that may be described by the family tree of a Crump–Mode–Jagers branching process stopped at a suitable time, including random recursive trees, preferential attachment trees, fragmentation trees and $m$-ary search trees. 

Janson~\cite{MR2908619} distinguishes three types of simply generated trees, numbered I, II and III, and for each the local limit exhibits a distinguishing characteristic. We use this terminology in our study of the vicinity of a random vertex. In the type I setting, the simply generated tree $\cT_n$ is distributed like a critical Galton--Watson tree conditioned on having $n$ vertices. Thus the height of a random vertex in $\cT_n$ is typically large and  extended fringe trees are typically small. In this regime, the limit is given by the random sin-tree constructed by Aldous~\cite{MR1102319}. Here the  word \emph{sin} refers to the fact  that, like the Kesten tree, this tree has almost surely up to finite initial segments only a single infinite path. When the offspring distribution has finite variance, we may even verify total variational convergence of the extended fringe subtree up to  $o(\sqrt{n})$-distant ancestors. 

While trees in the type I regime usually have small maximum degree, the types II and III are characterized by the appearance of vertices with large degree, which may be viewed as a form of condensation. Specifically, type II simply generated trees correspond to subcritical Galton--Watson trees with a heavy-tailed offspring distribution, and type III simply generated trees have superexponential branching weights such that no equivalent conditioned Galton--Watson tree exists. Our main contribution is in this condensation setting, where contrary to the type~I regime a random vertex may be near to the root, and extended fringe trees may have size comparable to the total number of vertices of $\cT_n$, as we are likely to encounter an ancestor with large degree. This is also a major difference to the settings addressed in the mentioned works by Aldous~\cite{MR1102319} and Holmgren and Janson~\cite{2016arXiv160103691H}.

We set up a compact space that encodes rooted plane trees that are centered around a second distinguished vertex, and establish several limit theorems. For arbitrary weight-sequences having type II or III, we establish a limit that   describes the vicinity of a random vertex up to and including its first ancestor with large degree. Here \emph{large} means having outdegree bigger than a deterministic sequence that tends to infinity sufficiently slowly.  The asymptotic shape of what lies beyond this ancestor appears to depend on the branching weights. In a way, the vertex with large degree obstructs the view to older generations.

We describe a novel limit object $\cT^*$ given by a random pointed plane tree, in which the pointed vertex has random distance from its first ancestor with infinite degree, and this ancestor again has a random number of ancestors with finite degree before the construction breaks off. For arbitrary weight-sequences, the asymptotic probability for the vicinity of a random vertex of $\cT_n$ to have a specific shape that admits at most one single ancestor of large degree, but allows ancestors with small degrees afterwards, coincides with the corresponding probability for the tree $\cT^*$. Our approach is based on a heavily modified depth-first-search to explore the tree $\cT_n$.  This yields information on how parts of a limit tree for the complete vicinity, that is not truncated at the first large ancestor, must look, if the simply generated tree $\cT_n$ pointed at  a random vertex converges weakly (along a subsequence). Note also that the compactness of the space, in which we formulate our limits, guarantees the existence of such subsequences. Thus the obstruction by the ancestor with large degree, that prevents us from seeing older generations, is not a complete blockage. However, this is not yet sufficient to deduce convergence in the space of pointed plane trees. In general, the tip of the backwards growing spine, where the construction of $\cT^*$ breaks off, may correspond to the root vertex of $\cT_n$, but just as well to a second ancestor with large degree. 

If the branching weights belong to a general regime of complete condensation, we manage to surpass the blockage and deduce weak convergence  toward~$\cT^*$. There are two main steps involved. First, we show that convergence toward $\cT^*$ is in fact equivalent to weak convergence of the height of a random vertex in $\cT_n$ to the height of the pointed vertex in the tree $\cT^*$, which in the type II regime is distributed like $1$ plus the sum of two independent identically distributed geometric random variables, and in the type III regime equals $1$. In this case, the root of $\cT^*$ really corresponds to the root of $\cT_n$. The second step verifies this property in the case of complete condensation, where the maximum degree of $\cT_n$ has the correct order.

In particular, Kortchemski's central limit theorem for the maximum degree~\cite[Theorem 1]{MR3335012} allows us to deduce convergence toward $\cT^*$ in the general case of a subcritical Galton--Watson tree conditioned on having $n$ vertices, if the offspring distribution $\xi$ satisfies \[\Pr{\xi = k} = f(k) k^{-\alpha}\] for a constant $\alpha >2$ and a function $f$ that varies slowly at infinity. 
In the type III regime where branching weights grow superexponentially fast, we consider the specific case where
\[
	\omega_k = k!^{\alpha}
\]
for $\alpha>0$. It is known that for these weights the maximum degree of $\cT_n$ has order $n + o_p(n)$, which may also be seen as complete condensation, see Janson, Jonsson,  and Stef{\'a}nsson~\cite{MR2860856} and Janson \cite[Example 19.36]{MR2908619}. Thus here the tree $\cT^*$ is also the weak limit of the simply generated tree $\cT_n$ pointed at a random vertex. There are, however, also   examples of superexponential branching weights that exhibit a more irregular behaviour \cite[Example 19.38]{MR2908619}, in which we are going to argue that weak convergence toward $\cT^*$ does not hold. 

\subsection*{Outline}
In Section~\ref{sec:notation} we fix basic notations, and Section~\ref{sec:simply} is dedicated to recall necessary background on simply generated trees. In Section~\ref{sec:space} we describe the metric space of rooted plane trees that are centered at a pointed vertex. This will be the setting in which we formulate our limit theorems.  In Section~\ref{sec:limits} we present our main results, and in  Section~\ref{sec:proofs} their proofs.

\section{Notation}
\label{sec:notation}

We let $\ndN$ denote the set of positive integers and set $\ndN_0 = \ndN \cup \{0\}$. The positive real numbers are denoted by $\ndR_{>0}$. Throughout, we usually assume that all considered random variables are defined on a common probability space.  The \emph{total variation distance} between two random variables $X$ and $Y$ with values in a countable state space $S$ is defined by
\[
d_{\textsc{TV}}(X,Y) = \sup_{\cE \subset S} |\Pr{X \in \cE} - \Pr{Y \in \cE}|.
\]
A sequence of real-valued random variables $(X_n)_{n \ge 1}$ is \emph{stochastically bounded}, if for each $\epsilon > 0$ there is a constant $M>0$ with
\[
\limsup_{n \to \infty} \Pr{ |X_n| \ge M} \le \epsilon.
\]
We denote this by $X_n = O_p(1)$. Likewise, we write $X_n = o_p(1)$ if the sequence converges to $0$ in probability. We use  $\convdis$ and $\convp$ to denote convergence in distribution and probability.
A function 
\[
h: \, \ndR_{>0} \to \ndR_{>0} \]is termed \emph{slowly varying}, if for any fixed $t >0$ it holds that
\[
\lim_{x \to \infty}\frac{h(tx)}{h(x)} = 1.
\]
Given sets $M$ and $N$, we let $N^M$ denote the set of all maps from $M$ to $N$. We also let $N^{(\ndN)}$ denote the set of all finite sequences of elements from $N$. 

\section{Simply generated trees}
\label{sec:simply}
A \emph{plane tree} is a rooted tree in which the offspring set of each vertex is endowed with a linear order. (Such trees are also sometimes referred to as planted plane trees or corner rooted plane trees, in order to distinguish them from related planar structures~\cite{MR2484382}.) Given a plane tree $T$ and a vertex $v \in T$ we let $d^+_T(v)$ denote its \emph{outdegree}, that is, the number of offspring. Its \emph{height} $\he_T(v)$ is its distance from the root-vertex.

We let $\mathbf{w} =(\omega_i)_{i \ge 0}$ denote a sequence of non-negative weights satisfying $\omega_0>0$ and $\omega_k >0$ for at least one $k \ge 2$.  The weight of a plane tree $T$ is defined by
\[
	\omega(T) = \prod_{v \in T} \omega_{d^+_T(v)}.
\]
The simply generated tree $\cT_n$ with $n$ vertices gets drawn from the set of all $n$-vertex plane trees with probability proportional to its weight. Galton--Watson trees conditioned on having a fixed  number of vertices are encompassed by this model of random plane trees. Of course, the tree $\cT_n$ is only well-defined if there is at least one plane tree with $n$ vertices that has positive weight. We set
\[
	\spa(\mathbf{w}) = \gcd\{i \ge 0 \mid \omega_i > 0 \}.
\]
As argued in \cite[Corollary 15.6]{MR2908619}, $n$-sized trees with positive weight may only exist for $n \equiv 1 \mod \spa(\mathbf{w})$, and conversely, they always exist if $n$ is large enough and belongs to this congruence class. We tacitly only consider such $n$ throughout this paper.

\subsection{Three types of weight-sequences}
\label{sec:types}
Janson~\cite[Chapter 8]{MR2908619} distinguishes three types of weight-sequences. The classification is as follows. Let $\rho_\phi$ denote the radius of convergence of the generating series
\[
	\phi(z) = \sum_{k \ge 0} \omega_k z^k.
\]
As argued in \cite[Lemma 3.1]{MR2908619}, if $\rho_\phi >0$ then the function
\[
	\psi(t) =  \phi'(t)t/\phi(t)
\]
admits a limit
\[
	\nu = \lim_{t \nearrow \rho_\phi} \psi(t) \in ]0, \infty]
\]
with the following properties. If $\nu \ge 1$, then there is a unique number $\tau$ with $\psi(\tau) = 1$ and we say the weight sequence $\mathbf{w}$ has type I. If $0 < \nu <1$, then we set $\tau := \rho_\phi < \infty$ and say $\mathbf{w}$ has type II. If $\rho_\phi = 0$, we say $\mathbf{w}$ has type III and set $\nu=0$ and $\tau = 0$.

The constant $\nu$ has a natural interpretation as the supremum of the means of all probability weight sequences equivalent to $\mathbf{w}$. The inclined reader may see \cite[Remark 4.3]{MR2908619}  for details.

\subsection{An associated Galton--Watson tree}
\label{sec:gwt}
We define the probability distribution $(\pi_k)_k$ on $\ndN_0$ by 
\begin{align}
	\label{eq:tt}
	\pi_k = \tau^k \omega_k / \phi(\tau).
\end{align}
The mean and variance of the distribution $(\pi_k)_k$ are given by
\begin{align}
	\label{eq:mu}
	\mu = \min(\nu,1)
\end{align}
and
\begin{align}
	\label{eq:variance}
	\sigma^2 = \tau \psi'(\tau) \le \infty.
\end{align}
We let $\xi$ denote a random non-negative integer with density $(\pi_k)_k$, and $\cT$ a Galton--Watson tree with offspring distribution $\xi$. Note that if $\mathbf{w}$ has type III, then $\xi=0$ almost surely and the tree $\cT$ consists of a single deterministic vertex.  As detailed in \cite[Section 4]{MR2908619}, if $\mathbf{w}$ has type I or II then the simply generated tree $\cT_n$ is distributed like the Galton--Watson tree $\cT$ conditioned on having $n$ vertices.

\section{The space of pointed plane trees}
\label{sec:space}

\subsection{Centering at a specified vertex}
The offspring of each vertex in a plane tree is  endowed with a linear order. We usually imagine a planar embedding where the root is at the top and the offspring of each vertex is ordered from the left to the right below it, ascendingly according to the corresponding linear order. Thus the "left-most" offspring is the minimum of the order. This is of course purely a matter of taste. Some prefer their trees to grow upwards, but regardless of the way for visualizing plane trees, we may use terms like height and depth-first-search in their usual sense without risk of confusion. In the present work we will also encounter plane trees that have no root, but whose vertex sets are endowed with a partial order that specifies the ancestry relations, and whose offspring sets are endowed with a linear order that is not required to have a smallest element.

Traditionally, plane trees are encoded as subtrees of the \emph{Ulam--Harris tree}. The Ulam--Harris tree $\UHT$ is an infinite plane tree with vertex set
\[
	\VHT = \ndN^{(\ndN)}
\]
given by the space of finite sequences of non-negative integers. Its root vertex is the unique sequence  with length zero, and the ordered offspring of a vertex $v$ are the concatenations $(v,i)$ for $i \ge 1$. Thus a plane tree is a subtree of the Ulam--Harris tree that contains its root, such that the offspring set of each vertex is an initial segment of the offspring of the corresponding vertex in $\UHT$.  Here we explicitly allow trees with  infinitely many vertices, and  vertices with countably infinite outdegree. If all outdegrees of a plane tree are finite, we say that it is \emph{locally finite}. The tree is \emph{finite}, if its total number of vertices is. We will usually let $o$ denote the root-vertex of a plane tree.

Subtrees of the Ulam--Harris tree are however not an adequate form to represent the vicinity of a specified vertex in a plane tree. If this vertex does not coincide with the root of the tree, then it has an ordered sequence of ancestors and possibly also siblings that lie to the left and right of it.  If we look at a random vertex of the simply generated $\cT_n$, then it may happen that the number of siblings to the left and/or right of it is asymptotically large, or that its distance from the root vertex is large. A sensible space in which we may describe the limit of the vicinity of the random vertex in $\cT_n$ must hence contain trees with a center that may have infinitely many ancestors, such that each may have infinitely many siblings to the left and/or right of it, including the center vertex itself.

For this reason, we describe the construction of an infinite tree $\UHT^\bullet$ that is embedded in the plane and has a spine $(u_i)_{i \ge 0}$ that grows "backwards". That is, we construct the tree $\UHT^\bullet$ by starting with an infinite path $u_0, u_1, \ldots$ $u_{i}$ of abstract vertices and define $u_{i}$ to be a parent of $u_{i-1}$ for all $i \ge 1$. Additionally, any vertex $u_i$ with $i \ge 1$ receives an infinite number of vertices to the left and to the right of its distinguished offspring $u_{i-1}$, and each of these "non-centered" offspring vertices is the root of a copy of the Ulam--Harris tree $\UHT$.  To conclude the construction, the start-vertex $u_0$ of the spine also gets identified with the root of a copy of $\UHT$.  We let $\VHT^\bullet$ denote the vertex-set of the tree $\UHT^\bullet$.  The tree $\UHT^\bullet$ is illustrated in Figure~\ref{fi:uht}. 

Note that the vertex set $\VHT^\bullet$ carries a natural partial order (given by the transitive hull of the parent-child relations specified in the construction of $\UHT^\bullet$), and the offspring set of any given vertex carries a natural linear order. This allows us to continue using the terms \emph{ancestor} and \emph{offspring} in this context.

The precise realization of the tree  $\UHT^\bullet$ will not be relevant for our arguments. One way to make its construction formal would be to define $\VHT^\bullet$ as a subset of $\ndN_0 \times \ndZ \times \ndN^{(\ndN)}$, with $u_i$ corresponding to $(i,0,\emptyset)$ for all $i \ge 0$, and $u_i$ having the linearly ordered offspring set $\{(i,j, \emptyset) \mid j \in \ndZ\}$ for all $i \ge 1$. Any point $(i,j, \emptyset)$ with either $(i,j) = (0,0)$, or $i \ge 1$ and $j \in \ndZ \setminus\{0\}$ is the root of a copy of the Ulam--Harris tree with vertex set $\{(i,j,v) \mid v \in \ndN^{(\ndN)}\}$.


A  plane tree $T$ together with a distinguished vertex $v_0$ is called a \emph{pointed} plane tree, and may be interpreted in a canonical way as a subtree of $\UHT^\bullet$. To do so, let $v_0, v_1, \ldots, v_k$ denote the path from $v_0$ to the root of $T$. This way, any vertex $v_i$ for $i \ge 1$ may have offspring to the left and to the right of $v_{i-1}$. Thus there is a unique order-preserving and outdegree preserving embedding of $T$ into $\UHT^\bullet$ such that $v_i$ corresponds to $u_i$ for all $0 \le i \le k$. Compare with Figure~\ref{fi:uht}.

\begin{figure}[t]
	\centering
	\begin{minipage}{1.0\textwidth}
		\centering
		\includegraphics[width=0.7\textwidth]{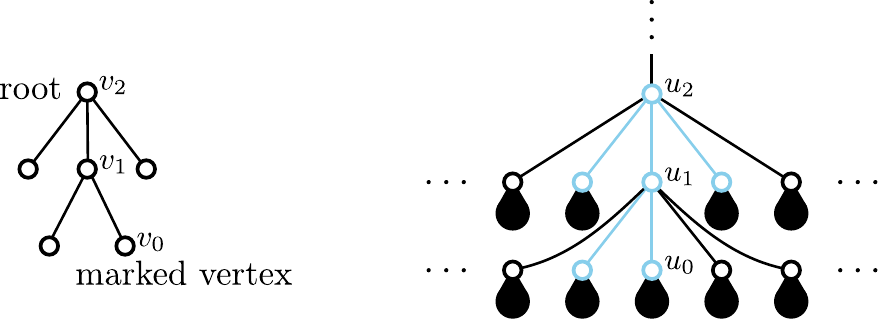}
		\caption{Embedding of a pointed plane tree into the tree $\UHT^\bullet$. Each black blob represents a copy of the Ulam--Harris tree.}
		\label{fi:uht}
	\end{minipage}
\end{figure}

\subsection{Topological properties}
Any plane tree $T$ may be identified with its family of outdegrees
\[
	(d^+_T(v))_{v \in \VHT} \in \overline{\ndN}_0^{\VHT},
\]
where we set $\overline{\ndN}_0 = \ndN_0 \cup \{\infty\}$. Here we use the  convention $d^+_T(v) = 0$ if $v \in \VHT$ is not a vertex of the tree $T$. We endow $\overline{\ndN}_0$ with the one-point compactification topology of the discrete space $\ndN_0$. Thus plane trees are elements of the compact product space $\overline{\ndN}_0^{\VHT}$. It is not hard to see that the subspace
\[
	\fmT \subset \overline{\ndN}_0^{\VHT}
\]
of all plane trees is closed.

Similarly, we may identify a pointed plane tree $T^\bullet = (T,v_0)$ with the corresponding family of outdegrees $(d^+_{T^\bullet}(v))_{v \in \VHT^\bullet}$, such that 
\begin{align}
	\label{eq:first}
	d^+_{T^\bullet}(v) \in \overline{\ndN}_0 
\end{align}
for $v \notin \{u_1, u_2, \ldots\}$, and
\begin{align}
	\label{eq:second}
d^+_{T^\bullet}(u_i) \in \{*\} \sqcup (\overline{\ndN}_0 \times  \overline{\ndN}_0), \qquad i \ge 1.
\end{align}
Here  the two numbers represent the number of offspring vertices to the left and right of the distinguished son $u_{i-1}$, and the $*$-placeholder represents the fact that the vertex has no offspring.

Since $\overline{\ndN}_0$ is a compact Polish space, so are the product  $\overline{\ndN}_0 \times  \overline{\ndN}_0$ and  the disjoint union topology on $\{*\} \sqcup (\overline{\ndN}_0 \times  \overline{\ndN}_0)$. Hence the space of all families $(d^+(v))_{v \in \cV^\bullet_\infty}$ satisfying
\[
 d^+(v) \in \begin{cases} \overline{\ndN}_0 &\text{for $v \notin \{u_1, u_2, \ldots\}$} \\ \{*\} \sqcup (\overline{\ndN}_0 \times  \overline{\ndN}_0) &\text{for $v \in \{u_1, u_2, \ldots\}$ } \end{cases}
\]
is the product of countably many compact Polish spaces, and hence also compact and Polish.  The subset $\fmT^\bullet$ of all elements that correspond to trees (that is, connected acyclic graphs) is closed, and hence also a compact Polish space with respect to the subspace topology.

Checking that  $\fmT^\bullet$ is closed is analogous to the arguments for $\fmT$: We may define the subset $\fmT^\bullet$ by certain local conditions.  In order for a family $(d^+(v))_{v \in \VHT^\bullet}$ in the product space to belong to $\fmT^\bullet$, we require for all $i \ge 1$ that $d^+(u_i) = *$ implies $d^+(u_{i+1}) = *$, and that $d^+(u_i) = (a,b) \in \overline{\ndN}_0 \times  \overline{\ndN}_0$ requires that $d^+(v)=0$ for all siblings $v$ of $u_{i-1}$ that lie more than $b$ to the right of $u_{i-1}$ or more than $b$ to the left of $u_{i-1}$. We furthermore require that for any vertex $u \in \VHT^\bullet \setminus \{u_1, u_2, \ldots\}$ with $d^+(u)=k \in \ndN$ it follows that $d^+(v) =0$ for all offspring vertices $v$ of $u$ that are placed more than $k$ to the right in the linear order of the offspring set of $u$. Now, if $\tau_n^\bullet$ is a deterministic sequence in the product space that converges toward a limit $\overline{\tau} \notin \fmT^\bullet$, then the limit must violate one of these conditions. But this implies that $\tau_n^\bullet$ violates this condition as well for all sufficiently large $n$. So, by contraposition,  $\fmT^\bullet$ must be  closed.

\section{The limit theorems}
\label{sec:limits}

As discussed in Section~\ref{sec:simply} there is a probability distribution $(\pi_k)_k$ associated with the weight sequence $\mathbf{w}$, with density given in \eqref{eq:tt}. Let $\xi$ \label{pp:xi} be distributed according to $(\pi_k)_k$ and let $\cT$ be a $\xi$-Galton--Watson tree. By Equation~\ref{eq:mu} it holds that \[\mu := \Ex{\xi}\le1.\] We may consider the {\em size-biased} random variable $\hat{\xi}$ with values in $\overline{\ndN}_0 $ and distribution given by
\[
\Pr{\hat{\xi} = k} = k \pi_k \quad \text{and} \quad \Pr{\hat{\xi} = \infty} = 1 - \mu.
\]
For any tree $T$ and any vertex $v\in T$ we let $f(T,v)$ denote the    \emph{fringe-subtree} of $T$ at $v$. That is, the maximal subtree of $T$ that is rooted at the vertex $v$.

Throughout the following, we let $v_0$ denote  a uniformly at random selected vertex of the simply generated plane tree $\cT_n$, that in the type I and II regime is distributed like the Galton--Watson tree $\cT$ conditioned on having $n$ vertices.

\subsection{The type I regime}
If the weight-sequence $\mathbf{w}$ has type I, then $\hat{\xi} < \infty$ almost surely, and we define the random pointed tree $\cT^{*}$ as follows.
Let $u_0$ be the root of an independent copy of the Galton--Watson tree $\cT$. For each $i \ge 1$, we let $u_i$ receive offspring according to an independent copy of $\hat{\xi}$. The vertex $u_{i-1}$ gets identified with an uniformly at random chosen offspring of $u_i$. All other offspring vertices of $u_i$ becomes the root of an independent copy of the Galton--Watson tree $\cT$. Compare with Figure~\ref{fi:ald}.

\begin{figure}[t]
	\centering
	\begin{minipage}{1.0\textwidth}
		\centering
		\includegraphics[width=0.15\textwidth]{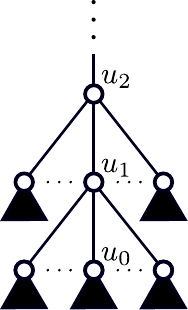}
		\caption{The limit tree $\cT^*$ in the type I regime. Each triangle represents an independent copy of the Galton--Watson tree $\cT$. For each $i \ge 1$ the vertex $u_i$ receives offspring according to an independent copy of $\hat{\xi}$, and the location of $u_{i-1}$ within that offspring set is chosen uniformly at random.}
		\label{fi:ald}
	\end{minipage}
\end{figure}  

\begin{theorem}
	\label{te:case1}
If the weight-sequence $\mathbf{w}$ has type I, then
	\[
		(\cT_n, v_0) \convdis \cT^{*}
	\]
	in the space $\fmT^\bullet$. 
\end{theorem}
Let $T$ be a plane tree, $v \in T$ a vertex, and $k \ge 0$ an integer.
If the vertex $v$ has a $k$th ancestor $v_k$, then we may define the pointed plane tree $H_k(T,v)$ as the fringe tree $f(T,v_k)$ that is rooted at the vertex $v_k$ and pointed at the vertex $v$. Here we use the term vertex in the graph-theoretic sense, since the coordinates of the vertex $v$ as node of the Ulam--Harris tree depend on whether we talk about $v \in T$ or $v \in f(T,v_k)$. If the vertex $v$ has height $\he_T(v) < k$, we set $H_k(T,v)=\diamond$ for some placeholder symbol $\diamond$.

\begin{theorem}
	\label{te:case1a}
	 Suppose that weight-sequence has type I and the offspring distribution $\xi$ has finite variance.  Let $k_n$ be an arbitrary sequence of non-negative integers that satisfies $k_n/\sqrt{n} \to 0$. Then 
	 \[
		 d_{\textsc{TV}}(H_{k_n}(\cT_n,v_0), H_{k_n}(\cT^*, u_0)) \to 0
	 \]
	 as $n$ becomes large.
\end{theorem}

Here we use the redundant notation $(\cT^*, u_0)$ to emphasize that the tree $\cT^*$ is marked at the vertex $u_0$.

\subsection{Complete condensation in the  type II regime}
\label{sec:compl2}
If the weight-sequence $\mathbf{w}$ has type II or III, then we construct  $\cT^{*}$ similarly as in the type I case, letting $u_0$ become the root of an independent copy of the Galton--Watson tree $\cT$, and letting for $i=1, 2, \ldots$ the vertex $u_i$ receives offspring according to an independent copy $\hat{\xi}_i$ of $\hat{\xi}$, where a uniformly at random chosen son gets identified with~$u_{i-1}$ (specifying the number of siblings to the left and right of $u_{i-1}$)  and the rest become roots of independent copies of $\cT$. We proceed in this way for $i=1,2, \ldots$ \emph{until} it occurs for the first time $i_1$ that $\hat{\xi}_{i_1}=\infty$. When $\hat{\xi}_1, \ldots, \hat{\xi}_{i_1-1}< \infty$ and $\hat{\xi}_{i_1} = \infty$, then $u_{i_1}$ receives infinitely many offspring to the left and right of its son $u_{i_1-1}$. Each of these vertices (except $u_{i_1-1}$ of course) gets identified with an independent copy of the Galton--Watson tree $\cT$. We then proceed as before for $i = i_1, i_1 + 1, \ldots$, such that $u_i$ receives offspring according to an independent copy $\hat{\xi}_i$ of $\hat{\xi}$, with a random son being identified with $u_{i-1}$ and the rest becoming roots of independent copies of $\cT$, \emph{until} it happens for the second time $i_2$ that $\hat{\xi}_{i_2}=\infty$. When $\hat{\xi}_{i_1} = \infty = \hat{\xi}_{i_2}$ for $i_1 < i_2$ and $\hat{\xi}_i < \infty$ for all $i < i_2$ with $i \ne i_1$, then we stop the construction. The spine of the resulting tree is then given by the ordered path $u_0, \ldots, u_{i_2-1}$. Compare with Figure~\ref{fi:con}.

\begin{figure}[t]
	\centering
	\begin{minipage}{1.0\textwidth}
		\centering
		\includegraphics[width=0.28\textwidth]{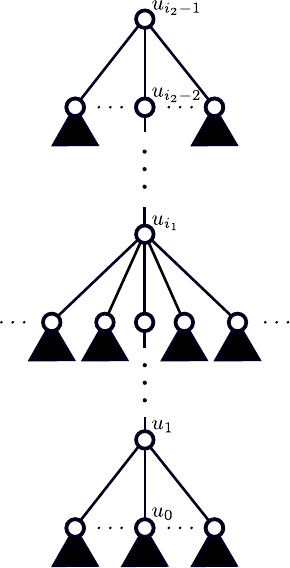}
		\caption{The limit tree $\cT^*$ in the complete condensation regime. The vertex $u_{i_1}$ is the only one having infinite degree, and each triangle represents an independent copy of the Galton--Watson tree $\cT$. }
		\label{fi:con}
	\end{minipage}
\end{figure}

\begin{theorem}
	\label{te:complete}
	Suppose that the weight-sequence $\mathbf{w}$ has type II. If the 
	maximum degree $\Delta(\cT_n)$ satisfies
		\[
		\Delta(\cT_n) = (1-\mu)n + o_p(n),
		\]
	then it holds that
	\[
	(\cT_n, v_0) \convdis \cT^{*} 
	\]
	in the space $\fmT^\bullet$. In particular, this is the case when there is a constant $\alpha>2$ and a slowly varying function $f$ such that for all $k$ \[\Pr{\xi=k} = f(k) k^{-\alpha}.\]
\end{theorem}
Here we make use of a  result by Kortchemski~\cite[Theorem 1]{MR3335012} who established a central limit theorem for the maximum degree, that ensures that $\Delta(\cT_n)$ has the correct order if the offspring distribution $\xi$ has a power law up to a slowly varying factor. There are also examples of offspring distributions with a more irregular behaviour. Janson~\cite[Example 19.37]{MR2908619} constructed a weight sequence such that along a subsequence $n = n_k$ it holds that $\Delta(\cT_{n}) = o_p(n)$, and along another subsequence several vertices with degree comparable to $n$ exist. This may be seen as incomplete condensation.

 The proof idea of Theorem~\ref{te:complete} is to deduce the asymptotic distribution of the height $\he_{\cT_n}(v_0)$ by localizing the vertex of $\cT_n$ having maximum degree at a position, that was also given in \cite[Theorem 2]{MR3335012} using results by Armend\'ariz and Loulakis~\cite{MR2775110} concerning
 conditioned random walks having a subexponential jump distribution. To do so, we employ results of Janson~\cite[Chapter 20]{MR2908619} that  (partially) use \[\Delta(\cT_n) = (1-\mu)n + o_p(n),\] but do not assume the offspring distribution to be subexponential. 
  The following main lemma, which characterizes convergence toward the tree $\cT^*$ in terms of weak convergence of the height $\he_{\cT_n}(v_0)$, then finalizes the proof of  Theorem~\ref{te:complete}.
\begin{lemma}
	\label{le:equ}
	If the weight-sequence $\mathbf{w}$ has type II or III, then the following three conditions are equivalent.
	\begin{enumerate}
		\item  $(\cT_n, v_0) \convdis \cT^{*}$ in $\fmT^\bullet$.
		\item $\he_{\cT_n}(v_0) \convdis \he_{\cT^{*}}(u_0)$.
		\item $\limsup_{n \to \infty} \Pr{\he_{\cT_n}(v_0) \ge k} \le \mu^k + k (1- \mu)\mu^{k-1}$ for all $k \ge 1$.
	\end{enumerate}
\end{lemma}
	Note that $\he_{\cT^{*}}(u_0)$ is distributed like $1$ plus the sum of two independent identically distributed geometric random variables that assume an integer $i$ with probability $\mu^i(1-\mu)$. 

\subsection{Complete condensation in the type III regime}
\label{sec:compl3}
If the weight-sequence $\mathbf{w}$ has type III, then it holds that $\mu=0$ and almost surely $\xi=0$ and $\hat{\xi} = \infty$. Here the Galton--Watson tree $\cT$ is always equal to a single point. Hence the tree $\cT^{*}$ is  obtained by letting $u_1$ have infinitely many offspring to the left and right of $u_0$, all of which (including $u_0$) are leaves. 

\begin{proposition}
	\label{pro:t3}
If the weight-sequence $\mathbf{w}$ has type III, then the following claims are equivalent.
\begin{enumerate}
	\item  $(\cT_n, v_0) \convdis \cT^{*}$ in $\fmT^\bullet$.
	\item $\he_{\cT_n}(v_0) \convp 1$.
	\item The maximum degree  $\Delta(\cT_n)$ satisfies $\Delta(\cT_n) = n + o_p(n)$.
\end{enumerate}
A general class of weight-sequences that demonstrate this behaviour is given by \[\omega_k = k!^{\alpha}\] with $\alpha>0$ a constant.
\end{proposition}

Here we have used that if $\omega_k = k!^\alpha$ with $\alpha>0$  a constant, then it is known \cite[Example 19.36]{MR2908619}, that the largest degree in $\cT_n$ has size $n + o_p(n)$.  But there are also other examples that exhibit a more irregular behaviour. In \cite[Example 19.38]{MR2908619} a weight-sequence is constructed such that along a subsequence $n=n_k$, for each $j \ge 1$ the $j$th largest degree $Y_{(j)}$ in  $\cT_{n_k}$ satisfies $Y_{(j)} = 2^{-j}$ with high probability. This may be seen as incomplete condensation. It is clear that in this case the limit of $(\cT_n, v_0)$, if it exists at all, must have a different shape than $\cT^{*}$. 

\subsection{Large nodes and truncated limits}
\label{sec:approx}
Suppose that the weight sequences $\mathbf{w}$ has type II or type III.
The limit theorems in Subsections~\ref{sec:compl2} and \ref{sec:compl3} work in settings of complete condensation, where the maximum degree of the tree $\cT_n$ satisfies
\[
	\Delta(\cT_n) = (1-\mu)n + o_p(n).
\]
If we content ourselves with the vicinity of the vertex $v_0$ up to and including the first vertex having large degree, we may obtain a limit theorem in complete generality. We are also going to construct a coupling to demonstrate how the vertex with infinite degree in the limit corresponds to a vertex with large degree in the simply generated tree $\cT_n$.

Janson~\cite[Lemma 19.32]{MR2908619} showed that there is a deterministic sequence $\Omega_n$ that tends to infinity sufficiently slowly, such that for any sequence $K_n \to \infty$ with $K_n \le \Omega_n$ it holds that 
the numbers $N_k$ of vertices with outdegree $k$ in the tree $\cT_n$ satisfy
\begin{align}
	\label{eq:om}
	\sum_{k \le K_n} k N_k = \mu n + o_p(n) \qquad \text{and} \qquad \sum_{k > K_n} k N_k = (1-\mu)n + o_p(n)
\end{align}
The sequence $\Omega_n$ may be replaced by any sequence that tends to infinity more slowly. Hence we may assume without loss of generality that $\Omega_n$ additionally satisfies
\begin{align}
	\label{eq:add}
	\Omega_n = o(n).
\end{align}
Let $\tilde{D}_n$ denote a random positive integer, that is independent from all previously considered random variables, with distribution given by 
\begin{align}
	\label{eq:dn}
	\tilde{D}_n \eqdist (d^+_{\cT_n}(o) \mid d^+_{\cT_n}(o) > \Omega_n).
\end{align}
Here we let $o$ denote the root-vertex of $\cT_n$. That is, $\tilde{D}_n$ is distributed like the root-degree conditioned to be "large".
We form the random tree $\bar{\cT}^*_{n}$ in a similar manner as the random tree $\cT^*$. The vertex $u_0$ becomes the root of an independent copy of the Galton--Watson tree $\cT$. For $i=1, 2, \ldots$ the vertex $u_i$ receives offspring according to independent copy $\hat{\xi}_i$ of $\hat{\xi}$, where a randomly chosen son gets identified with $u_{i-1}$ and the rest become roots of independent copies of $\cT$. We proceed in this way for $i=1,2, \ldots$ until it occurs that  $\hat{\xi}_{i}=\infty$. When $\hat{\xi}_1, \ldots, \hat{\xi}_{i-1}< \infty$ and $\hat{\xi}_{i} = \infty$, then $u_{i}$ receives $\tilde{D}_n$ offspring vertices, such that a uniformly at random chosen one gets identified with $u_{i-1}$, and the rest get identified with the roots of independent copies of $\cT$. Rather than continuing with the spine as in the construction of the tree $\cT^*$, we stop at this point, so that $u_{i}$ becomes the root of this tree. 

Given a pointed tree $T^\bullet = (T,v)$ and an ancestor $a$ of $v$, we let $f^\bullet(T^\bullet,a)$ denote the fringe subtree of $T$ at $a$ that we consider as pointed at the vertex $v$. We refer to $f^\bullet(T^\bullet,a)$ as the \emph{pointed fringe subtree} of the pointed tree $T^\bullet$ at the vertex $a$.

Let $v_0$ denote a uniformly at random selected vertex of the simply generated tree $\cT_n$. Let $H(\cT_n, v_0, \Omega_n)$ denote the pointed fringe subtree of $(\cT_n, v_0)$ at the youngest ancestor of $v_0$ that has outdegree bigger than $\Omega_n$. If no such vertex exists (which is unlikely to happen, as we are going to verify), set $H(\cT_n,v_0,\Omega_n)=\diamond$ for some fixed placeholder value $\diamond$. 

\begin{theorem}
	\label{te:trunc}
	Suppose that the weight sequence $\mathbf{w}$ has type II or III.
	Let $\bar{\cT}^*$ denote the pointed fringe subtree of the tree $\cT^*$ at its unique vertex with infinite degree.
	Then it holds that
	\[
	H(\cT_n,v_0,\Omega_n) \convdis \bar{\cT}^*.
	\]
	in the space $\fmT^\bullet$.
\end{theorem}

The strength of this theorem is its generality, as we make no additional assumptions on the weight-sequence at all.  It is suitable for applications where it is not necessary to look behind the large vertex. 

We may still improve upon this.  For each $n$, let $\cT_n^*$ be constructed from $\bar{\cT}^*$ by pruning  at its root vertex such that its outdegree becomes $\tilde{D}_n$. Of course we have to select one of the $\tilde{D}_n$ ways of how much we prune from the left and right so that the total outdegree becomes $\tilde{D}_n$, and we choose an option uniformly at random.

For each integer $m \ge 0$ we let $\bar{V}^{[m]} \subset \VHT^\bullet$ denote the vertex set of the tree obtained from $\UHT^\bullet$ by deleting all vertices with distance larger than $m$ from the center vertex $u_0$ and pruning so that the vertices $u_i$, $1 \le i \le m$ have outdegree $(m,m)$ and the remaining vertices all have outdegree  equal to $m$.  The topology on the subspace $\fmT^\bullet_{\mathrm{lf}} \subset \fmT^\bullet$  of locally finite trees is induced by the metric
\[
	d_{\fmT^\bullet_{\mathrm{lf}}}(T_1^\bullet,T_2^\bullet) = 1 / \sup\{m \ge 0 \mid d_{T_1^\bullet}^+(v) = d_{T_2^\bullet}^+(v) \text{ for all $v \in \bar{V}^{[m]}$} \}.
\]
This can be verified using the fact that a sequence $(T_n)_n$ in $\fmT^\bullet$ converges towards an element $T\in \fmT^\bullet$ if and only if $d_{T_n}^+(v)$ converges towards $d_{T}^+(v)$ for each $v \in \VHT^\bullet$.

\begin{theorem}
	\label{te:approx}
	Suppose that the weight sequence $\mathbf{w}$ has type II or III.
	For any finite set of vertices $x_1, \ldots, x_r \in \VHT^\bullet$ it holds that
	\[
		d_{\textsc{TV}}( (d_{H(\cT_n,v_0,\Omega_n)}^+(x_i))_{1 \le i \le r}, (d^+_{\bar{\cT}^*_n}(x_i))_{1 \le i \le r}) \to 0.
	\]
	Equivalently, there is a coupling of $(\cT_n, v_0)$ and $\bar{\cT}^*_n$ such that
	\[
		d_{\fmT^\bullet_{\mathrm{lf}}}( H(\cT_n,v_0,\Omega_n), \bar{\cT}^*_n) \convp 0.
	\]
\end{theorem}

In Equation (20.4) and the subsequent paragraph of \cite{MR2908619}, Janson also argues that if 
\[
\Delta(\cT_n) = (1 - \mu)n + o_p(n),
\]
then
\[
d_{\textsc{TV}}(\Delta(\cT_n), \tilde{D}_n) \to 0.
\]
Hence in the complete condensation regime, we may choose $\tilde{D}_n$ in the coupling of Theorem~\ref{te:approx}  such that $\tilde{D}_n = \Delta(\cT_n)$ with probability tending to $1$ as $n$ becomes large. This yields the asymptotic location of the vertex with maximum degree with respect to the random vertex $v_0$.

In a way, we could say that the ancestor of the random vertex $v_0$ with large degree in $\cT_n$ blocks the view, if we are not in the complete condensation regime. As we shall see in Lemma~\ref{le:general} and Lemma~\ref{le:condens} it does not block it  completely though. Roughly  said, for each fixed $k \ge 1$ we obtain the asymptotic probability for the event, that the pointed fringe subtree  at the $k$th ancestor of the random vertex $v_0$ has a given shape that involves at most one large vertex on the spine.  This yields more information than Theorem~\ref{te:trunc} on how a limit tree $\bar{\cT}$ must look, if $(\cT_n, v_0)$ converges weakly (along a subsequence). (Note that the compactness of the space $\fmT^\bullet$ guarantees the existence of such subsequences.) It also  suggests that if the center vertex of $\bar{\cT}$ has almost surely at most one ancestor with infinite degree, then it must already hold that $\bar{\cT} \eqdist \cT^*$, but we leave it to the inclined reader to pursue this line of thought further.

\section{Proof of the main results}
\label{sec:proofs}

\subsection{Preliminaries}

\subsubsection{Simply generated trees and balls in boxes}
\label{sec:preli}
For any integers $m, n \ge 1$ we may consider the balls-in-boxed model $(Y_i^{(m,n)})_{0 \le i <n}$ that randomly draws a vector vector  $(y_0, \ldots, y_{n-1})$ of non-negative integers satisfying \[\sum_{i=0}^{n-1} y_i = m\] with probability proportional to $\omega_{y_0} \cdots \omega_{y_{n-1}}$. To shorten notation, we set
\[
	Y_i = Y_i^{(n-1,n)}
\]
for all $i$, as this will be the case that we will consider most of the time. This model is related to the outdegree sequence $(d_0, \ldots, d_{n-1})$ of the simply generated tree $\cT_n$ by
\begin{align}
\label{eq:balls}
(d_0, \ldots, d_{n-1}) \eqdist ((Y_0, \ldots, Y_{n-1}) \mid \sum_{i=0}^\ell (Y_i -1) \ge 0 \text{ for all $0 \le \ell < n-1$} ).
\end{align}
Here we may form the outdegree sequence according to depth-first-search order, but many other form of vertex explorations are possible. In general, consider the following family of algorithms, that order the vertices of finite deterministic plane trees.
\begin{enumerate}
	\item Take a plane tree $T$ as input.
	\item Let $\cP$ denote the ordered list of visited vertices, that initially is empty. Let $\cQ$ denote the ordered queue of vertices that are scheduled to be visited next, that we initialize with the root of $T$.
	\item Move the first vertex $v$ of the ordered queue $\cQ$ to the end of the list $\cP$ of visited vertices. We then modify the queue $\cQ$ of vertices that are scheduled to be visited next  so that it additionally contains all the sons of the vertex $v$ in the tree $T$. We do so by a fixed rule, that may take into account the current state of $\cQ$ and $\cP$, and need not respect the previous order of vertices in $\cQ$ or the order of the offspring vertices of $v$.
	\item We repeat the third step until the queue $\cQ$ of scheduled vertices is empty.
\end{enumerate}
If we order the vertices of the random tree $\cT_n$ according to an algorithm of this form, then the corresponding sequence of outdegrees satisfies Equation~\eqref{eq:balls}. This degree of freedom will be crucial in our analysis of extended fringe subtrees of $\cT_n$ in the condensation regime.

A classical combinatorial result (see for example \cite[Cor 15.4]{MR2908619}) states that for each  vector $(y_i)_{0 \le i \le n-1}$ of numbers $y_i \ge -1$ with 
$
\sum_{i=0}^{n-1} y_i = -1 
$ there is a unique cyclic shift 
\begin{align}
\label{eq:shift}
(z_0, \ldots, z_{n-1}) = (y_{j \mod n},\, y_{j+1 \mod n},\, \ldots,\, y_{j + n -1 \mod n})
\end{align} by $0 \le j \le n-1$, such that for all $0 \le \ell < n-1$ it holds that $\sum_{i=0}^\ell (z_i -1) \ge 0.$ 
Thus there is a coupling of the simply generated tree $\cT_n$ with the balls-in-boxes model $(Y_0, \ldots, Y_{n-1})$ such that the degree sequence $(d_0, \ldots, d_{n-1})$ of $\cT_n$ is a random cyclic shift of  $(Y_0, \ldots, Y_{n-1})$.

\subsubsection{Large nodes near the root in the condensation regime}
\label{sec:large}
Suppose that the weight-sequence $\mathbf{w}$ has type II or type III.
Let $\xi$ denote the offspring distribution defined in Section~\ref{sec:simply} and $\mu = \Ex{\xi} \in [0, 1[$ its first moment. Let $\tilde{D}_n$ be the random non-negative integer defined in Equation~\eqref{eq:dn}. Consider the random variable $\tilde{\xi}$ defined by
\[
	\Pr{\tilde{\xi} = k} = k \Pr{\xi=k}
\]
for $k \in \ndN$, and 
\[
\Pr{\tilde{\xi} = \diamond} = 1- \mu
\]
for some placeholder $\diamond$. Janson~\cite[Chapter 20]{MR2908619} defined the following modified Galton--Watson tree $\hat{\cT}_{1n}$. There are normal and special vertices, and we start with a special root. Each normal vertex receives offspring according to an independent copy of $\xi$, all its sons are also normal. For any special vertex we consider an independent copy of $\tilde{\xi}$. If it is a finite number, then we add accordingly many offspring and declare a uniformly at random selected son as special. If it assumes the placeholder value $\diamond$, then we add offspring according to $\tilde{D}_n$, all of which are normal. Thus the tree $\hat{\cT}_{1n}$ comes with an almost surely finite spine whose tip $v^*$ satisfies
\[
	\Pr{\he_{\hat{\cT}_{1n}}(v^*) = k} = \mu^k (1-\mu)
\]
for all $k \ge 0$. Janson~\cite[Theorem 20.2]{MR2908619} showed that for any finite list of vertices $v_1, \ldots, v_\ell \in \VHT$ it holds that
\begin{align}
\label{eq:tvconv}
	d_{\textsc{TV}}( (d^+_{\cT_n}(v_i))_{1 \le i \le \ell}, (d^+_{\hat{\cT}_{1n}}(v_i))_{1 \le i \le \ell}) \to 0
\end{align}
as $n$ becomes large. 

For each integer $m \ge 0$ we let $V^{[m]} \subset \VHT$ denote the vertex set of the tree obtained from $\UHT$ by truncating at height $m$ and pruned so that all out-degrees are equal to $m$. That is, $V^{[m]}$ corresponds to all sequences of non-negative integers with length at most $m$ such that each element of the sequences less than or equal to $m$. The topology on the space of locally finite plane trees is induced by the  metric
\[
\delta_1(T,T') := 1 / \sup\{m \ge 1 \mid (d_T^+(v))_{v \in V^{[m]}} = (d_{T'}^+(v))_{v \in V^{[m]}} \}.
\]
The convergence~\eqref{eq:tvconv} is equivalent to the existence of a coupling of the random trees $\cT_n$ and $\hat{\cT}_{1n}$ such that
\begin{align}
\label{eq:couple}
	\delta_1( \cT_n,  \hat{\cT}_{1n}) \convp 0.
\end{align}

\subsection{Convergence in the  type I regime}

We present a proof of Theorem~\ref{te:case1a}, in which we make use of properties that are characteristic of the type I regime, where extended fringe subtrees are typically small and random vertices have typically large height. Theorem~\ref{te:case1}, which states convergence of $(\cT_n, v_0)$ toward $\cT^*$ if $\mathbf{w}$ has type I, follows directly from a general observation, given in  Lemma~\ref{le:general} below, that is valid for weight sequences having arbitrary type.

\begin{proof}[Proof of Theorem~\ref{te:case1a}]
Suppose that the weight sequence $\mathbf{w}$ has type I, and that the offspring distribution $\xi$ has finite variance $\sigma^2$.	 By assumption it holds that $k_n = n^{1/2}t_n$ for some sequence $t_n \to 0$. Without loss of generality, we may assume that $k_n \to \infty$. 

For any $k$ let $\cE_{k,n}$ denote the set of all pairs $(T,x)$ of a plane tree $T$ having at most $n t_n$ vertices and a vertex $x$ having height $\he_T(x) = k$, such that $\Pr{H_k(\cT_n, v_0) = (T,x)} > 0$. We are going to argue that as $n \equiv 1 \mod \spa(\mathbf{w})$ becomes large
		\begin{enumerate}[\qquad i)]
			\item $\Pr{ H_{k_n}(\cT^*, u_0) \in \cE_{k_n, n}} \to 1$,
			\item $\Pr{ H_{k_n}(\cT_n, v_0) \in \cE_{k_n, n} } \to 1$,
			\item $\sup_{(T,x) \in \cE_{k_n, n}} |\Pr{H_{k_n}(\cT_n, v_0)  = (T,x) } / \Pr{  H_{k_n}(\cT^*, u_0) = (T,x) } - 1 | \to 0$.
		\end{enumerate}
		This suffices, as i) and ii) imply that
		\begin{multline}
		d_{\textsc{TV}}( H_{k_n}(\cT_n, v_0), H_{k_n}(\cT^*, u_0)) = \\o(1) + \sup_{\cH \subset \cE_{k_n, n}} | \Pr{H_{k_n}(\cT_n, v_0)\in \cH} - \Pr{H_{k_n}(\cT^*, u_0) \in \cH}|,
		\end{multline}
		and this expression converges to zero by iii).
		
		We start with property i). We have to show that $H_{k_n}(\cT^*, u_0)$ has with high probability at most $n t_n$ vertices. The size of $H_{k}(\cT^*, u_0)$ is given by the sum of $|H_{0}(\cT^*, u_0)| \eqdist |\cT|$ and the  independent differences \[|H_{i}(\cT^*, u_0)| - |H_{i-1}(\cT^*, u_0)| \eqdist 1 + S_{\hat{\xi}-1}, \quad i=1\ldots k,\] with $\hat{\xi}$ the size-biased version of the offspring distribution $\xi$, and \[S_m = X_1 + \ldots + X_m\] a sum of independent copies $(X_j)_j$ of $|\cT|$. The reason for this is that $H_{i}(\cT^*, u_0)$ is given by the root-vertex $u_i$, with the tree $H_{{i-1}}(\cT^*, u_0)$ and $d^+_{\cT^*}(u_i) -1\eqdist \hat{\xi} -1$ independent copies of $\cT$ dangling from it. Compare with Figure~\ref{fi:ald}.
		
		Hence  \[
		|H_{k}(\cT^*, u_0)| \eqdist k + S_{ M_k - (k-1)}
		\] is stochastically bounded by the sum $S_{M_{k+1}}$ with \[M_i = Z_1 + \ldots + Z_{i}\] the sum of $i$ independent copies of $\hat{\xi}$ for all $i \ge 1$. (Here we have used that $\hat{\xi} \ge 1$ by definition.) By a general result for the size of Galton--Watson forests, there is a constant $C>0$ such that 
		\[
		\Pr{S_m \ge x} \le C m x^{-1/2}
		\]
		for all $m$ and $x$. See Devroye and Janson \cite[Lem. 4.3]{MR2829308} and Janson \cite[Lem. 2.1]{MR2245498}. We assumed that $\sigma^2< \infty$, hence $\hat{\xi}$ has a finite first moment. It follows that
		\[
		\Pr{|H_{k}(\cT^*, u_0)| \ge x } \le \Pr{S_{M_{k+1}} \ge x} \le C \Ex{M_{k+1}} x^{-1/2} = C (k+1) \Ex{\hat{\xi}} x^{-1/2}.
		\]
		Setting $x=n t_n$ and $k=k_n = n^{1/2}t_n$, it follows by $k_n \to \infty$ and $t_n \to 0$ that \[\Pr{|H_{k_n}(\cT^*, u_0)| \ge nt_n} \le C (n^{1/2}t_n + 1) \Ex{\hat{\xi}} (n t_n)^{-1/2} = o(1).
		\]  This verifies i).

		Property ii) is actually a consequence of properties i) and iii). Indeed, iii) implies that 
		\[
		\Pr{  H_{k_n}(\cT_n, v_0) \in \cE_{k_n,n}} - \Pr{  H_{k_n}(\cT^*, u_0) \in \cE_{{k_n},n}} \to 0,
		\]
		and by i) it follows that
		\[
		\Pr{  H_{k_n}(\cT_n, v_0) \in \cE_{k_n,n}} \to 1.
		\]
		
		It remains to verify iii). Let $(T,x) \in \cE_{k_n,n}$. Given $\cT_n$, there is a one to one correspondence between the vertices $v \in \cT_n$ with fringe subtree $f(\cT_n, v) = T$, and the vertices $v'$ with $H_{k_n}(\cT_n, v') = (T,x)$. Thus
		\begin{align}
			\label{eq:t1}
			\Pr{H_{k_n}(\cT_n, v_0) = (T,x)} = \Pr{f(\cT_n,v_0) = T}.
		\end{align}
		The fringe subtree $f(\cT^*, u_{k_n})$ is distributed like the modified Galton--Watson tree, in which there are two types of vertices, normal and special, and we start with a special root. Normal vertices receive offspring according to an independent copy of $\xi$ and all of those are normal again. Special vertices receive offspring according to an independent copy of $\hat{\xi}$, and one of them is selected uniformly at random and declared its heir. If the heir has height less than $k_n$, then it is declared special, and otherwise it becomes a normal vertex. Here the unique heir that is not special corresponds to the vertex $u_0$. The probability for a special vertex to have $\ell$ offspring such that precisely the $i$th is selected as heir is given by
		\[
			\Pr{\hat{\xi} = \ell} / \ell = \Pr{\xi=\ell}.
		\]
		Hence
		\begin{align}
			\label{eq:t2}
			\Pr{ H_{k_n}(\cT^*, u_0) = (T,x)} = \Pr{\cT = T}.
		\end{align}
		Combining Equations~\eqref{eq:t1} and \eqref{eq:t2} yields
		\[
			\Pr{H_{k_n}(\cT_n, v_0) = (T,x)} / 	\Pr{ H_{k_n}(\cT^*, u_0) = (T,x)} = \Pr{f(\cT_n,v_0) = T} / \Pr{\cT = T}.
		\]
		Remark 15.8, Equation (17.1) and subsequent equations in Janson's survey \cite{MR2908619} yield that
		\[
		\Pr{f(\cT_n,v_n) = T_k} / \Pr{ \cT = T_k} = \Pr{\overline{S}_{n-|H_k|} = 0} /  \Pr{\overline{S}_{n} = -1}
		\]
		with $\overline{S}_\ell$ denoting the sum of $\ell$ independent copies of $\xi -1$. Since $(T,x) \in \cE_{k_n,n}$, the tree $T$ has at most $n t_n$ vertices. Consequently, the local limit theorem for sums of lattice distributed random variables \cite[Ch. 3.5]{MR2722836} yields that uniformly for all $(T,x) \in \cE_{k_n,n}$ as $n \equiv 1 \mod \spa(\mathbf{w})$ becomes large
		\[
		\Pr{\overline{S}_{n-|T|} = 0} /  \Pr{\overline{S}_{n} = -1} = (1+o(1)) \frac{o(1) + \frac{\spa(\mathbf{w})}{\sqrt{2\pi \sigma^2}}} {o(1) + \frac{\spa(\mathbf{w})}{ \sqrt{2 \pi \sigma^2} } \exp(\frac{-1}{ 2n \sigma^2})} = 1 + o(1).
		\]
		This verifies iii) and hence completes the proof.
\end{proof}

\subsection{General observations}
\label{sec:genob}

We state two observations that are valid for weight sequences having arbitrary type. The first describes the asymptotic probability to encounter small extended fringe subtrees.

\begin{lemma}
	\label{le:general}
	Let $v_0$ be a uniformly at random selected vertex of the simply generated plane tree $\cT_n$. Let $T^\bullet$ be a pointed plane tree whose pointed vertex has height $h \ge 0$. Then the probability, that the pointed fringe subtree at the $h$th ancestor $v_h$ of the random vertex $v_0 \in \cT_n$ is equal to the pointed tree $T^\bullet$, converges to the probability, that the pointed fringe subtree of $\cT^{*}$ at the spine vertex $u_h$ is equal ot $T^\bullet$. That is, with $\cT_n^\bullet = (\cT_n, v_0)$ it holds that
	\[
		\Pr{f^\bullet(\cT_n^\bullet, v_h) = T^\bullet} \to \Pr{f^\bullet(\cT^*, u_h) = T^\bullet}.
	\]
	Here we consider $\cT^*$ as pointed at the center $u_0$.
	If we let $V \subset \VHT^\bullet$ denote the subtree of the modified Ulam--Harris tree $\UHT^\bullet$ that corresponds to the pointed tree $T^\bullet$, then this may be expressed by
	\[
	\Pr{ d^+_{\cT_n^\bullet}(v) = d^+_{T^\bullet}(v) \text{ for all $v \in V$}} \to \Pr{ d^+_{\cT^{*}}(v) = d^+_{T^\bullet}(v) \text{ for all $v \in V$}}.
	\]
\end{lemma}

Here $f^\bullet(\cdot, \cdot)$ denotes the pointed fringe subtree as defined in Section~\ref{sec:approx}.
The following result describes the asymptotic probability for extended fringe subtrees containing an ancestor with large degree.

\begin{lemma}
	\label{le:condens}
	Let $v_0$ be a uniformly at random selected vertex of the simply generated plane tree $\cT_n$. Let $T^\bullet$ be a pointed plane tree whose pointed vertex has height $h \ge 0$.
	For all integers $1 \le k \le h$ and all sufficiently large $\ell$ we may consider the event, that the outdegrees of the pointed fringe subtree of the $h$th ancestor of the random vertex $v_0 \in \cT_n$ all agree with the outdegrees of $T^\bullet$, except for the $k$th ancestor of $v_0$, which is required to have at least $\ell$ offspring to the left and at least $\ell$ offspring to the right of its unique son that is also an ancestor of $v_0$. As $n$ becomes large, this probability converges toward an expression that depends on $\ell$. If we let $\ell$ tend to infinity, then this expression converges toward the probability, that the outdegrees of the fringe subtree of $\cT^{*}$ at the spine vertex $u_h$ agree with the outdegrees of $T^\bullet$, except for $u_k$, which must have outdegree $(\infty, \infty)$. In other words, $u_k$ is required to have an infinite number of offspring vertices to the left and to the right of $u_{k-1}$. Expressed in more formal words, let $V \subset \VHT^\bullet$ denote the subtree of $\UHT^\bullet$ that corresponds to the tree $T^\bullet$.  Then the event $\cE(\ell,n)$ that
\[
d^+_{\cT_n^\bullet}(v) = d^+_{T^\bullet}(v)
\] for all $v \in V \setminus \{u_i\}$ and
\[
d^+_{\cT_n^\bullet}(u_k) \in \{\ell, \ell+1, \ldots\} \times \{\ell, \ell+1, \ldots\} 
\]
satisfies
\begin{multline*}
\lim_{\ell \to \infty} \lim_{n \to \infty} \Pr{\cE(\ell,n)} = \\  \Pr{d^+_{\cT^{*}}(v) = d^+_{T^\bullet}(v) \text{ for all $v \in V \setminus \{ u_k\}$}, d^+_{\cT^{*}}(u_k) = (\infty, \infty) }.
\end{multline*}
\end{lemma}

These results certainly deserve some explanation. If the weight sequence $\mathbf{w}$ has type I, then Lemma~\ref{le:general} immediately yields weak convergence of $(\cT_n, v_0)$ toward $\cT^{*}$. This proves Theorem~\ref{te:case1}. Aldous \cite{MR1102319} showed a similar form of convergence for the case where $\mathbf{w}$ has type I and the associated offspring distribution has finite variance, and Janson~\cite[Thm. 7.12]{MR2908619} established convergence of the fringe subtree at $v_0$ for arbitrary weights. The proof of Lemma~\ref{le:general} uses this result and various others from \cite{MR2908619}.

In the type II and III setting, the situation is more complicated and Lemma~\ref{le:condens} is not sufficient to deduce convergence in $\fmT^\bullet$ for arbitrary weight-sequences. It is  intuitive, that a random vertex is only likely to be close to the root, if one of its ancestors has large degree. Lemma~\ref{le:condens} provides a description of what happens near a random root up to its first ancestor that has large degree. Beyond that, we only obtain information on what happens in the case that the ancestors of this ancestor have small degree. Since the space $\fmT^\bullet$ is compact, the sequence $(\cT_n, v_n)$ clearly converges weakly toward a limit along a subsequence, and the distribution of this limit must agree by Lemma~\ref{le:condens} with  $\cT^{*}$ until the point where the spine of $\cT^{*}$ stops. However, at this location, we could encounter the root vertex of the limit, but just as well a second ancestor with large degree.

\begin{proof}[Proof of Lemma~\ref{le:general}]
	Let $v_0, v_1, \ldots$ be the directed path from $v_0$ to the root of $\cT_n$. Let $T^\bullet = (T,v)$ be a pointed, finite plane tree and let $h$ denote the height of the vertex $v$ in $T$. Consider the event that $v_0$ has height at least $h$, and that the pointed fringe subtree $f^\bullet(\cT_n^\bullet, v_h)$ of $\cT_n^\bullet = (\cT_n, v_0)$ is equal to $T^\bullet$. Given $\cT_n$, there is a one to one correspondence between the vertices $v$ with fringe subtree $f(\cT_n, v) = T$ and the vertices $v'$ whose $h$-th ancestor $u$ has pointed fringe subtree $f^\bullet((\cT_n,v'), u) = T^\bullet$. Thus,
	\begin{align}
	\label{eq:ff1}
	\Pr{f^\bullet(\cT_n^\bullet, v_h) = T^\bullet} = \Pr{f(\cT_n, v_0) = T}.
	\end{align}
	Janson \cite[Thm. 7.12]{MR2908619} showed that \begin{align}
	\label{eq:ff2}
	\lim_{n \to \infty}  \Pr{f(\cT_n, v_0) = T} = \Pr{\cT = T}.
	\end{align}
	The probability for the size-biased random variable $\hat{\xi}$ to assume a value $\ell$, and that a uniformly at random choice out of $\ell$ options yields a specific value $\ell_0$, is equal to the probability that $\xi$ equals $\ell$. Thus
	\[
	\Pr{\cT = T} = \Pr{f^\bullet(\cT^{*}, u_h) = T^\bullet}.
	\]
	Combined with Equations~\eqref{eq:ff1} and \eqref{eq:ff2} this yields that
	\begin{align}
	\label{eq:ff3}
	\lim_{n \to \infty} \Pr{f^\bullet(\cT_n^\bullet, v_h) = T^\bullet} = \Pr{f^\bullet(\cT^{*}, u_h) = T^\bullet}.
	\end{align}
	This proves the first claim.
\end{proof}

\begin{proof}[Proof of Lemma~\ref{le:condens}]
	Let $T^\bullet = (T,v)$ denote a finite pointed plane tree, where the pointed vertex is not equal to its root.
	Let $k_1, k_2$ be arbitrary non-negative integers whose sum is larger than the maximum degree of the tree $T$. This assumption will be crucial in the following argument.
	
	Suppose that $o$ is a vertex that lies on the path from the root to the pointed vertex of $T^\bullet$, but is not equal to the pointed vertex. In order to not confuse the three vertices, let us call the root of $T$ the inner root, the pointed vertex of $T^\bullet$ the outer root, and the vertex $o$ the middle root.
	
	Let $\cE_{k_1, k_2}(T^\bullet)$ denote the set of pointed plane trees obtained by connecting the root vertices of $k_1$ arbitrary plane-trees from the left to the middle root $o$ of $T^\bullet$, and $k_2$ from the right. We are interested in the event that $f^\bullet(\cT_n^\bullet, v_k) \in \cE_{k_1, k_2}(T^\bullet)$. Let $\cE_{k_1, k_2}(T)$ denote the corresponding set where we forget about which outer-root (the pointed vertex) we distinguished. Equation~\eqref{eq:ff1} yields
	\begin{align}
	\label{eq:gg1}
	\Pr{f^\bullet(\cT_n^\bullet, v_k) \in \cE_{k_1, k_2}(T^\bullet)} = \Pr{f(\cT_n, v_0) \in \cE_{k_1, k_2}(T)}.
	\end{align}
	In order to study this asymptotic probability, we make use of a modified depth-first search of the tree.
	
	Traditional depth-first-search (DFS) lists the vertices of a plane tree by starting with the root, and traverses in each step along the left-most previously unvisited son. If no such son exists, we go to the parent of the current vertex and try again. The process terminates with an ordered list of all vertices of the tree. Note that at any time the search maintains an ordered list of vertices that it already visited, and an ordered list $\cQ$ of vertices that are scheduled to be visited next. Anytime we visit a new vertex that is not a leaf, vertices are added to the front of the queue $\cQ$ of vertices that are to be visited next.
	
	Let $K$ be the sum of $k_1$, $k_2$ and the outdegree  $d_T^+(o)$ of the middle root of $T$.  We may modify the DFS  by treating vertices with out-degree $K$ in a special manner. When we encounter such a vertex, instead of putting all its offspring in front of the queue $\cQ$, we put the $(k_1+1)$th to $(k_1 + d_T^+(o))$th offspring to the front of the queue $\cQ$, and the remaining offspring to the back. 
	Thus, if none of the fringe subtrees of the vertices we put to the front of the queue has a vertex with degree $K$,  we traverse next along the $(k_1+1)$th son its entire fringe subtree, and so on, until the $(k_1 + d_T^+(o))$ son and its entire fringe subtree. After this we proceed with the remaining siblings of $o$.
	
	As we assumed that $K$ is larger than the maximum degree of $T$, this means that if we search a tree $T'$ from $\cE_{k_1, k_2}(T)$, the first $|T|$ vertices in the resulting list of ordered vertices correspond to the vertices of $T$, and their outdegrees are equal to those in $T$, except for the vertex $o$, which has outdegree $d^+_{T'}(o) = K$.
	
	We now proceed similarly as in the proof of Janson's result \cite[Thm. 7.12]{MR2908619} where classical DFS was used. Let $d_0, \ldots, d_{n-1}$ denote the list of outdegrees in the simply generated tree $\cT_n$ according to the modified DFS-order. For any $i \ge n$ we set $d_i = d_{i \mod n}$. Let $\hat{d}_0, \ldots, \hat{d}_{\ell}$ denote the DFS-ordered list of outdegrees in $T$, and let $i_0$ denote the unique index that corresponds to the vertex $o$. Since $K$ is larger than the maximal outdegree of $T$, the vertices $v$ of $\cT_n$ with fringe-subtree in $\cE_{k_1, k_2}(T)$ correspond bijectively to the indices $0 \le i \le n-1$ with 
	\[
	(\hat{d}_0, \ldots, \hat{d}_{i_0 -1}, K, \hat{d}_{i_0 + 1}, \ldots, \hat{d}_\ell) = (d_i, d_{i+1}, \ldots, d_{i+\ell}).
	\]
	This explicitly includes the case where $i$ is so close to $n-1$ such that $i+\ell > n-1$. It is not possible for a tree to have an ending segment in its list of vertices that is equal to an initial segment of $(K,\hat{d}_1,  \ldots, \hat{d}_\ell)$, because then the search of the tree would have terminated with a non-empty queue $\cQ$ of vertices that still need to be visited.

	Consider the balls-in-boxes model $(Y_0, \ldots, Y_{n-1})$ from Equation~\eqref{eq:balls}. We set $Y_i = Y_{i \mod n}$ for $i \ge n$. For all $0 \le j \le n-1$ let $I_j$ be the indicator for the event
	\[
	(Y_{j}, \ldots, Y_{j+\ell}) = (\hat{d}_0, \ldots, \hat{d}_{i_0 -1}, K, \hat{d}_{i_0 + 1}, \ldots, \hat{d}_\ell).
	\]
	The sum $\sum_{j=0}^{n-1} I_j$ is rotational invariant, hence it follows from Equation~\eqref{eq:shift} that
	\begin{align*}
	\Pr{f(\cT_n, v_0) \in \cE_{k_1, k_2}(T)} &= \Ex{n^{-1} \sum_{j=0}^{n-1} I_j}  
	= \Ex{I_0} \\
	&= \Pr{(Y_0, \ldots, Y_\ell) = (\hat{d}_0, \ldots, \hat{d}_{i_0 -1}, K, \hat{d}_{i_0 + 1}, \ldots, \hat{d}_\ell)}.
	\end{align*}
	For ease of notation, we define
	\[
		(\bar{d}_0, \bar{d}_1, \ldots, \bar{d}_\ell) := (\hat{d}_{i_0}, \hat{d}_0, \ldots, \hat{d}_{i_0 -1}, \hat{d}_{i_0 + 1}, \ldots, \hat{d}_\ell).
	\]
	By exchangeability, it follows that
	\[
		\Pr{f(\cT_n, v_0) \in \cE_{k_1, k_2}(T)} = \Pr{(Y_0, \ldots, Y_\ell) = (K, \bar{d}_1, \ldots, \bar{d}_\ell)}.
	\]
	Combining this with Equation~\eqref{eq:gg1} yields
	\begin{align}
	\label{eq:gg2}
	\Pr{f^\bullet(\cT_n^\bullet, v_k) \in \cE_{k_1, k_2}(T^\bullet)} = \Pr{(Y_0, \ldots, Y_\ell) = (K,\bar{d}_1,  \ldots, \bar{d}_\ell)}.
	\end{align}
	Setting
	\[
	\cE_{\ge k_1, \ge k_2} = \bigcup_{\ell_1 \ge k_1, \ell_2 \ge k_2} \cE_{\ell_1, \ell_2},
	\]
	it follows that
	\begin{align}
	\label{eq:schuessel}
	\Pr{f^\bullet(\cT_n^\bullet, v_k) \in \cE_{\ge k_1, \ge k_2}(T^\bullet)} &= \sum_{r \ge K}  \sum_{\substack{\ell_1 + \ell_2 + \bar{d}_0 = r \\ \ell_1 \ge k_1, \ell_2 \ge k_2}} \Pr{(Y_0, \ldots, Y_\ell) = (r,\bar{d}_1,  \ldots, \bar{d}_\ell)} \nonumber \\
	&= \sum_{r \ge K} (r-K+1)\Pr{(Y_0, \ldots, Y_\ell) = (r,\bar{d}_1,  \ldots, \bar{d}_\ell)}.
	\end{align}
	For any $j \ge 0$, let $N_j$ denote the number of indices $0 \le i \le n-1$ with $Y_i = j$. Conditioned on the $N_j$, the numbers $Y_0, Y_1, \ldots$ are obtained by placing $N_0$ $0$'s, $N_1$ $1$'s, \ldots, in uniformly random order. So, as stated in \cite[Eq. (14.44)]{MR2908619} (with the slight notational difference that Janson labelled the boxes from $1$ to $n$ rather than from $0$ to $n-1$), it follows that for $r \ge K$
	\begin{align}
		\label{eq:s0}
	\Pr{(Y_0, \ldots, Y_\ell) = (r,\bar{d}_1,  \ldots, \bar{d}_\ell) \mid N_0, N_1, \ldots} = \frac{N_r}{n} \prod_{i=1}^\ell \frac{N_{\bar{d}_i} - c_i }{n-i}
	\end{align}
	with $c_i$ denoting the number of $1 \le j <i$ with $\bar{d}_j = \bar{d}_i$. (Here we have used the fact that $K$ (and hence also $r$) is larger than $\bar{d}_1, \ldots, \bar{d}_\ell$.) Hence
	\begin{align}
	\label{eq:schuessel2}
	\Pr{(Y_0, \ldots, Y_\ell) = (r,\bar{d}_1,  \ldots, \bar{d}_\ell)} &= \Exb{ \frac{N_r}{n} \prod_{i=1}^\ell \frac{N_{\bar{d}_i} - c_i }{n-i}} \nonumber \\
	 &= \Exb{\frac{N_r}{n} \prod_{i=1}^\ell \frac{N_{\bar{d}_i}}{n} + O\left(\frac{N_r}{n^2}\right)}, 
	\end{align}
	where the implicit constant in the $O$ term does not depend on $n$ or $r$.   It follows by Equation \eqref{eq:schuessel} that
	\begin{multline*}
	\Pr{f^\bullet(\cT_n^\bullet, v_k) \in \cE_{\ge k_1, \ge k_2}(T^\bullet)} 
	= \\ \Exb{\frac{\sum_{r \ge K} rN_r}{n} \prod_{i=1}^\ell \frac{N_{\bar{d}_i}}{n} + (1-K) \frac{\sum_{r \ge K} N_r}{n} \prod_{i=1}^\ell \frac{N_{\bar{d}_i}}{n} + O\left( \frac{K \sum_{r \ge K} rN_r}{n^2}\right)}.
	\end{multline*}	
	It holds that \[\sum_{j \ge 1} j N_j = n-1 \qquad \text{and} \qquad \sum_{j \ge 0} N_j = n.\] Hence
	\begin{multline*}
	\Pr{f^\bullet(\cT_n^\bullet, v_k) \in \cE_{\ge k_1, \ge k_2}(T^\bullet)} 
	= \\ \Exb{\frac{n-1 - \sum_{r < K} rN_r}{n} \prod_{i=1}^\ell \frac{N_{\bar{d}_i}}{n} + (1-K) \frac{n - \sum_{r < K} N_r}{n} \prod_{i=1}^\ell \frac{N_{\bar{d}_i}}{n}} + O\left( \frac{K}{n}\right).
	\end{multline*}	
	By Janson's result \cite[Thm. 11.4]{MR2908619} it holds for each fixed $j$ as $n$ becomes large that
	\[
	N_j/n \convp \Pr{\xi = j}.
	\]
	Thus, by dominated convergence, it follows that as $n$ becomes large
	\begin{align}
	\label{eq:schuessel3}
	\Pr{f^\bullet(\cT_n^\bullet, v_k) \in \cE_{\ge k_1, \ge k_2}(T^\bullet)} \to (\Pr{\hat{\xi} \ge K} + (1-K) \Pr{\xi \ge K}) \prod_{i=1}^\ell \Pr{\xi = \bar{d}_i}.
	\end{align}
	Recall that the first moment of $\xi$ is given by \[
	\mu = \min(1, \nu) \in [0,1].
	\]
	Clearly it holds that
	\[
	\Pr{\hat{\xi} \ge K} + (1-K) \Pr{\xi \ge K} \le \Pr{\xi \ge K} + \Pr{\hat{\xi} \ge K} \to 1- \mu
	\]
	as $K$ becomes large. As for a lower bound, we may write for every $0 < \epsilon < 1$
	\begin{align*}
	\Pr{\hat{\xi} \ge K} + (1-K) \Pr{\xi \ge K} &\ge \Pr{\xi \ge K} + \sum_{k \ge K} (k - K) \Pr{\xi = k} \\
	&\ge \Pr{\xi \ge K} + \sum_{k > K/\epsilon} (1-\epsilon)k \Pr{\xi=k} \\
	&\to (1-\epsilon)(1-\mu)
	\end{align*}
	as $K$ becomes large. As $\epsilon>0$ was arbitrary, it follows that
	\[
	\Pr{\hat{\xi} \ge K} + (1-K) \Pr{\xi \ge K} \to 1 - \mu
	\]
	as $K$ tends to infinity. Hence Equation~\eqref{eq:schuessel3} implies that for any sequences $k_1(r)$ and $k_2(r)$ with $k_1(r) + k_2(r) \to \infty$ as $r$ becomes large it holds that
	\[
	\lim_{r \to \infty} \lim_{n \to \infty} \Pr{f^\bullet(\cT_n^\bullet, v_k) \in \cE_{\ge k_1(r), \ge k_2(r)}(T^\bullet)} \to (1-\mu)\prod_{i=1}^\ell \Pr{\xi = \bar{d}_i}.
	\]
	Since
	\[
		\Pr{d^+_{\cT^{*}}(v) = d^+_{T^\bullet}(v) \text{ for all $v \in V \setminus \{ u_k\}$}, d^+_{\cT^{*}}(u_k) = (\infty, \infty) } = (1-\mu)\prod_{i=1}^\ell \Pr{\xi = \bar{d}_i},
	\]
	this concludes the proof.
\end{proof}

\subsection{The limit theorems in the condensation regime}

\subsubsection{The type II regime}
\begin{proof}[Proof of Lemma~\ref{le:equ}]
	Suppose that the weight-sequence $\mathbf{w}$ has type II or III. 
	We need to show that the following three statements are equivalent.
	\begin{enumerate}
		\item  $(\cT_n, v_0) \convdis \cT^{*}$.
		\item $\he_{\cT_n}(v_0) \convdis \he_{\cT^{*}}(u_0)$.
		\item $\limsup_{n \to \infty} \Pr{\he_{\cT_n}(v_0) \ge k} \le \mu^k + k (1- \mu)\mu^{k-1}$ for all $k \ge 1$.
	\end{enumerate}
	It is clear that the first claims implies the second, since the height 
	\[
		\he: \fmT^\bullet \to \overline{\ndN}_0, \quad (T,x) \mapsto \he_T(x)
	\] is a continuous functional on the space $\fmT^\bullet$. The height $\he_{\cT^*}(u_0)$ of the pointed vertex in $\cT^*$ is distributed like $1$ plus the sum of two independent identically geometric random variables with parameter $\mu$. Thus the second claim implies the third. The convergence in Lemma~\ref{le:condens} immediately yields that for all $t \ge 1$
	\[
	\liminf_{n \to \infty} \Pr{ \he_{\cT_n}(v_0) \ge t} \ge   \Pr{\he_{\cT^{*}}(u_0) \ge t} = \mu^k + k (1- \mu)\mu^{k-1}.
	\]
	Hence the third claim implies the second. It remains to verify that the second claim implies the first. Suppose that 
	\begin{align}
	\label{eq:co}
	\he_{\cT_n}(v_0) \convdis \he_{\cT^{*}}(u_0).
	\end{align}
	Since the space $\fmT^\bullet$ is compact, any sequence of random pointed plane trees has a convergent subsequence. In particular, the sequence $(\cT_n, v_0)$   converges toward a limit object $\bar{\cT}$ along a subsequence $(n_k)_k$. We are going to show that 
	\begin{align}
		\label{eq:final}
	\bar{\cT} \eqdist \cT^*
	\end{align}
	 regardless of the subsequence. By standard methods \cite[Thm. 2.2]{MR0310933} this implies \[(\cT_n, v_0) \convdis \cT^*.\]
	 By Equation~\eqref{eq:co} it holds that
	 \begin{align}
		 \label{eq:height}
		 \he_{\bar{\cT}}(u_0) \eqdist \he_{\cT^*}(u_0).
	 \end{align}
Lemma~\ref{le:general} yields that for any finite tree $T^\bullet = (T,x) \in \fmT^\bullet$ with $\he_T(u_0) = k$ it holds that
	 \begin{align}
		 \label{eq:eq1}
		 \Pr{f^\bullet(\cT^*, u_k) = T^\bullet} = \Pr{f^\bullet(\bar{\cT}, u_k) = T^\bullet}.
	 \end{align}
	 By Lemma~\ref{le:condens} we know furthermore that for any index $1 \le i \le k$ it holds that
	 \begin{multline}
		 \label{eq:eq2}
		 \Pr{ d^+_{\cT^*}(u_i) = (\infty, \infty), d^+_{\cT^*}(v) = d_{T^\bullet}^+(v) \text{ for all $v \in V \setminus\{u_i\}$} } = \\
		 \Pr{ d^+_{\bar{\cT}}(u_i) = (\infty, \infty), d^+_{\bar{\cT}}(v) = d_{T^\bullet}^+(v) \text{ for all $v \in V \setminus\{u_i\}$} }
	 \end{multline}
	 with $V \subset \VHT^\bullet$ denoting the subset corresponding to the vertices of $T^\bullet$.
	 
	 We are going to show that Equations~\eqref{eq:height}, \eqref{eq:eq1} and \eqref{eq:eq2} are sufficient to verify that $\bar{\cT} \eqdist \cT^*$. The first step is to verify that
	 \begin{align}
		 \label{eq:step1}
		 (d^+_{\cT^*}(u_i))_{i \ge 1} \eqdist (d^+_{\bar{\cT}}(u_i))_{i \ge 1}
	 \end{align}
	 as random elements of the product space
	 \[
		 \left( \{*\} \sqcup (\overline{\ndN}_0 \times \overline{\ndN}_0) \right)^{\ndN}.
	\]
	 For this, it is sufficient to verify that for all $k \ge 1$
	 \[
	 (d^+_{\cT^*}(u_i))_{1 \le i \le k} \eqdist (d^+_{\bar{\cT}}(u_i))_{1 \le i \le k}.
	 \]
	 To this end, let
	 \[
		 d_1, \ldots, d_k \in \{*\} \sqcup (\overline{\ndN}_0 \times \overline{\ndN}_0)
	 \]
	 be given, such that there exists an index $0 \le j \le k$ such that for all $i > j$ it holds that $d_i = *$. We are going to show that
	 \begin{align}
		 \label{eq:intermediate1}
		 \Pr{ d^+_{\cT^*}(u_i) = d_i, 1 \le i \le k} = \Pr{ d^+_{\bar{\cT}}(u_i) = d_i, 1 \le i \le k}.
	 \end{align}
	 This suffices, as the set of indices $i$ with $d^+_{\cT^*}(u_i) = *$ must form a tail-segment of $(1, \ldots, k)$, since $\cT^*$ is a tree, and likewise for $\bar{\cT}$. Moreover, Equation~\eqref{eq:height} implies that almost surely \[d^+_{\bar{\cT}}(u_1), d^+_{\cT^*}(u_1) \ne *.\] Hence we may additionally assume that
	 $
		 j \ge 1.
	 $
	 
	 First, let us observe that
	 \begin{align}
		 \label{eq:careful}
		 \Pr{ d^+_{\cT^*}(u_i) = d_i, 1 \le i \le j} \le \Pr{ d^+_{\bar{\cT}}(u_i) = d_i, 1 \le i \le j}.
	 \end{align}
	 Indeed, the left-hand side is equal to zero unless  $d_i \in \ndN_0\times \ndN_0$ for all $1 \le i \le j$ with the exception of at most one index $i_0$ for which we allow that $d_{i_0} = (\infty, \infty)$. If the $d_i$ satisfy this property, we may argue as follows.
	 We constructed the tree $\cT^*$ in a way such that for all vertices $v \in \VHT^\bullet \setminus \{u_1, u_2, \ldots\}$ the fringe-subtree $f(\cT^*,v)$ is finite.  Hence the event $d^+_{\cT^*}(u_i) = d_i, 1 \le i \le j$ is a countable disjoint union of events of the form considered in Equations~\eqref{eq:eq1} and \eqref{eq:eq2}. That is, if all $d_i \in  \ndN \times \ndN$ for $1 \le i \le j$, then Inequality~\eqref{eq:careful} follows by applying Equation~\eqref{eq:eq1} for countably many finite pointed trees $T^\bullet \in \fmT^\bullet$. If $d_{i_0} = (\infty, \infty)$ for an index $1 \le i_0 \le j$, then Inequality~\eqref{eq:careful} follows by applying Equation~\eqref{eq:eq2} for countably many finite trees $T ^\bullet\in \fmT^\bullet$ with $d_{T^\bullet}^+(u_{i_0}) = (0,0)$. 
	 
	 Thus Inequality~\eqref{eq:careful} holds. If we sum over all $d_1, \ldots, d_j \in \overline{\ndN}_0 \times \overline{\ndN}_0$, then the left-hand side of \eqref{eq:careful} sums up to $\Pr{\he_{\cT^*}(u_0) \ge j}$, and the right-hand side to $\Pr{\he_{\bar{\cT}}(u_0) \ge j}$. But these two quantities are equal by Equation~\eqref{eq:height}. Thus it follows that already
	 \begin{align}
	 \label{eq:careful2}
	 \Pr{ d^+_{\cT^*}(u_i) = d_i, 1 \le i \le j} = \Pr{ d^+_{\bar{\cT}}(u_i) = d_i, 1 \le i \le j}
	 \end{align}
	 for all $d_1, \ldots, d_j \in \overline{\ndN}_0 \times \overline{\ndN}_0$. 
	 
	 If $j = k$, then Equation~\eqref{eq:careful2} is identical to Equation~\eqref{eq:intermediate1}. Otherwise, if $1 \le j <k$, then Equation~\eqref{eq:careful2} implies that
	 \begin{align*}
		 &\Pr{ d^+_{\cT^*}(u_i) = d_i, 1 \le i \le k} \\ &\,\,=  \Pr{ d^+_{\cT^*}(u_{j+1}) = *, d^+_{\cT^*}(u_i) = d_i \text{ for all } 1 \le i \le j} \\
		 &\,\,= \Pr{ d^+_{\cT^*}(u_i) = d_i, 1 \le i \le j } 
		 - \sum_{d \in \overline{\ndN}_0 \times \overline{\ndN}_0} \Pr{d^+_{\cT^*}(u_{j+1}) = d, d^+_{\cT^*}(u_i) = d_i, 1 \le i \le j}.
	 \end{align*}
	 Of course, the same holds if we replace $\cT^*$ by $\bar{\cT}$. It follows by Equation~\eqref{eq:careful2}, that the last expression is equal for $\cT^*$ and $\bar{\cT}$. This verifies Equation~\eqref{eq:intermediate1}, and hence also Equation~\eqref{eq:step1}.
	 
	 In order to verify that $\cT^* \eqdist \bar{\cT}$, we may proceed in a similar manner. Letting $d_1, \ldots, d_k$ and $1 \le j \le k$ be as before, Equations~\eqref{eq:eq1}, \eqref{eq:eq2} and \eqref{eq:step1} imply that
	 \begin{align}
		 \label{eq:step2}
		 (f^\bullet(\cT^*, u_j) \mid d^+_{\cT^*}(u_i) = d_i, 1 \le i \le j) \eqdist (f^\bullet(\cT_n^*, u_j) \mid d^+_{\bar{\cT}}(u_i) = d_i, 1 \le i \le j).
	 \end{align}
	 Indeed, if $d_1, \ldots, d_j$ are finite, then there are only countably many values that the pointed fringe tree \[
	 T_1 := (f^\bullet(\cT^*, u_j) \mid d^+_{\cT^*}(u_i) = d_i, 1 \le i \le j)
	 \]
	  may assume.  By Equation~\eqref{eq:eq1} and \eqref{eq:step1}, the tree
	  \[
		  T_2 := (f^\bullet(\cT_n^*, u_j) \mid d^+_{\bar{\cT}}(u_i) = d_i, 1 \le i \le j)
	  \]
	  assumes each with the same probability as $T_1$, so it follows that  the tree $T_2$ is also supported on a countable set and $T_1 \eqdist T_2$. As for the other case, suppose that $d_{i_0} = (\infty, \infty)$ for a unique index $1 \le i_0 \le j$. For each $\ell \ge 0$ we may look at the canonically ordered finite list $L_\ell(T_1)$ of fringe subtrees at the sons $v \ne u_1, \ldots, u_{j-1}$ of the  $u_i$ for $i \ne i_0$ and at the first $\ell$ siblings to the left and to the right of $u_{j-1}$. Again there are only countably many outcomes for $L_\ell(T_1)$, as each of these fringe trees must be finite.  By Equations~\eqref{eq:eq2} and \eqref{eq:step1}, the list $L_\ell(T_2)$ assumes each with the same probability. Hence $L_\ell(T_2)$ is also supported on a countable set and  $L_\ell(T_1) \eqdist L_\ell(T_2)$. As this holds for arbitrary $\ell$, it follows that  $T_1 \eqdist T_2$.  Hence Equation~\eqref{eq:step2} holds.
	  
	  Letting $d_1, \ldots, d_j$ range over all allowed values, it follows from Equations~\eqref{eq:step1} and \eqref{eq:step2} that 
	 \begin{align}
	\label{eq:step3}
		(f^\bullet(\cT^*, u_j) \mid \he_{\cT^*}(u_0) \ge j) \eqdist (f^\bullet(\cT_n^*, u_j) \mid \he_{\bar{\cT}}(u_0) \ge j).
\end{align}
	In order to deduce that $\cT^* \eqdist \bar{\cT}$, we need to show that any Borel-measurable set $\cE \subset \fmT^\bullet$  and any $h \ge 1$ it holds that
	\begin{align}
		\label{eq:step4}
		\Pr{f^\bullet(\cT^*, u_h) \in \cE, \he_{\cT^*}(u_0) = h} = \Pr{f^\bullet(\cT_n^*, u_h) \in \cE, \he_{\bar{\cT}}(u_0) = h}.
	\end{align}
	Clearly it suffices to show this when $\cE$ contains only trees $T \in \fmT^\bullet$ with $\he_T(u_0) = h$. In this case, it follows by Equation~\eqref{eq:step3} that
	\begin{align*}
		&\Pr{f^\bullet(\cT^*, u_h) \in \cE, \he_{\cT^*}(u_0) = h} \\ 
				&\quad= 	\Pr{f^\bullet(\cT^*, u_h) \in \cE} - 	\Pr{f^\bullet(\cT^*, u_h) \in \cE, \he_{\cT^*}(u_0) \ge h+1} \\
		&\quad= 	\Pr{f^\bullet(\cT^*, u_h) \in \cE, \he_{\cT^*}(u_0) \ge h} - 	\Pr{f^\bullet(\cT^*, u_h) \in \cE, \he_{\cT^*}(u_0) \ge h+1} \\
		&\quad= 	\Pr{f^\bullet(\bar{\cT}, u_h) \in \cE, \he_{\bar{\cT}}(u_0) \ge h} - 	\Pr{f^\bullet(\bar{\cT}, u_h) \in \cE, \he_{\bar{\cT}}(u_0) \ge h+1} \\
		&\quad=\Pr{f^\bullet(\bar{\cT}, u_h) \in \cE, \he_{\bar{\cT}}(u_0) = h}.
	\end{align*}
	This verifies Equation~\eqref{eq:step4} and hence completes the proof.
\end{proof}

\begin{proof}[Proof of Theorem~\ref{te:complete}]
	Suppose that the weight sequence $\mathbf{w}$ has type II  and that the maximum degree $\Delta(\cT_n)$ has order
	\begin{align}
	\label{eq:assumption}
	\Delta(\cT_n) = (1-\mu)n + o_p(n).
	\end{align}
	By Kortchemski's central limit theorem for $\Delta(\cT_n)$~\cite[Theorem 1]{MR3335012}, we know that this holds for example when $\omega_k = f(k) k^{-\alpha} \rho_\phi^{-k}$ for a constant $\alpha >2$ and a slowly varying function $f$.
	In order to show that 
	\[
		(\cT_n, v_0) \convdis \cT^*,
	\]
	it suffices by Lemma~\ref{le:equ} to show that
	\[
				\he_{\cT_n}(v_0) \convdis \he_{\cT^*}(u_0).
	\]
	
	Let $\tilde{D}_n$ denote the random integer defined in Equation~\eqref{eq:dn} by
	\begin{align*}
		\tilde{D}_n \eqdist (d^+_{\cT_n}(o) \mid d^+_{\cT_n}(o) > \Omega_n).
	\end{align*}
	for any fixed deterministic sequence $\Omega_n$ that tends to infinity slowly enough such that Equation~\eqref{eq:om} holds. Here $o \in \cT_n$ denotes the root-vertex of the tree $\cT_n$. Let $\hat{\cT}_{1n}$ denote the modified Galton--Watson tree constructed in Subsection~\ref{sec:large}. Janson~\cite[Equation (20.2)]{MR2908619} argued that it follows from the assumption \eqref{eq:assumption} that
	\begin{align}
		\Pr{d_{\cT_n}^+(o) = \Delta(\cT_n)} = 1 - \mu + o(1).
	\end{align}
	By Equations~\eqref{eq:om}, \eqref{eq:balls} and \eqref{eq:shift} it holds that
	\begin{align}
		\Pr{d_{\cT_n}^+(o) > \Omega_n} = 1 - \mu + o(1).
	\end{align}
	Using Equation~\eqref{eq:add} it follows that
	\begin{align}
		d_{\textsc{TV}}( \tilde{D}_n, \Delta(\cT_n)) \to 0
	\end{align}
	as $n$ becomes large. In particular there is a sequence $\delta_n \to 0$ such that
	\begin{align}
		\label{eq:cc}
	\tilde{D}_n \in  (1 - \mu \pm \delta_n )n
	\end{align}
	with high probability. For any plane tree $T$ let $(F(T),v(T))$ denote the pointed plane tree obtained by marking the first vertex with outdegree larger than $\Omega_n$ in the depth-first-search ordered list of vertices of $T$ and cutting away all its descendants. The convergence in~\eqref{eq:couple} implies that
	\begin{align}
		\label{eq:dd}
		d_{\textsc{TV}}\left( (F(\cT_n), v(\cT_n), d_{\cT_n}^+(v(\cT_n))), (F(\hat{\cT}_{1n}),v(\hat{\cT}_{1n}), \tilde{D}_n) \right) \to 0.
	\end{align}
	Note that the distribution of $(F(\hat{\cT}_{1n}),v(\hat{\cT}_{1n}))$ does not depend on $n$ and that $|F(\hat{\cT}_{1n})|$ is stochastically bounded.  
	 Given $(F(\cT_n), v(\cT_n), d_{\cT_n}^+(v(\cT_n)))$ the the ordered list of fringe subtrees dangling from the vertex $v(\cT_n)$ in $\cT_n$ are conditionally distributed like a simply generated forest with $d_{\cT_n}^+(v(\cT_n))$ trees and $n - |F(\cT_n)|$ of vertices. That is, if $(F(\cT_n), v(\cT_n))$ is equal to some finite pointed plane tree $F$ and $d_{\cT_n}^+(v(\cT_n)) = \ell$ for some integer $\ell \in (1 - \mu \pm \delta_n)n$ then the list is distributed like the forest \[
	 \left(\cT^1, \ldots, \cT^\ell \mid \sum_{1 \le i \le \ell} |\cT^i| = n - |F|\right)
	 \] with $(\cT^i)_{i \ge 1}$ denoting independent copies of the $\xi$-Galton--Watson tree~$\cT$. Let $(\xi_i)_{i \ge 1}$ denote independent copies of the offspring distribution $\xi$. It follows from~\cite[Lem. 15.3, Thm. 18.1]{MR2908619} that
	 \begin{align}
	 	\label{eq:ff}
	 	\Pr{\sum_{1 \le i \le \ell} |\cT^i| &= n - |F|} = \Pr{\sum_{i=1}^{n - |F|}(\xi_i -1) = -\ell, \sum_{i=1}^k(\xi_i -1) > -\ell \text{ for all $k<|F|$}} \nonumber \\
	 	&= \frac{\ell}{n - |F|} \Pr{\xi_1 + \ldots + \xi_{n - |F|}  = n - |F| -\ell} \nonumber \\
	 	&= \exp(o(1))
	 \end{align}
	 uniformly for all $\ell \in (1 - \mu \pm \delta_n)n.$  For each $k \ge 0$, let
	 \[
		\ell_k: \fmT \to \bar{\ndN}_0,
	 \]
	 denote the  map that sends a tree to its width at height $k$. By the Azuma--Hoeffding inequality it follows that for any integer $r \ge 1$ and any $\epsilon>0$ there are constants $C,c>0$ such that
	 \begin{align}
	 	\label{eq:gg}
	 	\Pr{ |\{ 1 \le i \le \ell \mid \ell_k(\cT^i) = r\}| - \ell \Pr{\ell_k(\cT)=r}| > \epsilon n} \le C \exp(-cn).
	 \end{align}
	 By~\eqref{eq:ff} and \eqref{eq:gg} it follows that for any fixed integer $R\ge1$ 
	 \[
	 	\Pr{ n^{-1} \sum_{1 \le i \le \ell} \ell_k(\cT^i)\one_{\ell_k(\cT^i)} \notin \Ex{\ell_k(\cT) \one_{\ell_k(\cT) \le R}} \pm \epsilon} \to 0
	 \]
	uniformly for all $\ell \in (1 - \mu \pm \delta_n)n.$ As $k$, $R$ and $F$ where fixed but arbitrary, this implies together with \eqref{eq:cc} and \eqref{eq:dd} that for any $t \ge 1$
	\begin{align*}
		\liminf_{n \to \infty} \Pr{\he_{\cT_n}(v_0) = t} &\ge \sum_{i=0}^{t-1} \Pr{\he_{F(\hat{\cT}_{1n})}(v(\hat{\cT}_{1n})) = i} \Ex{\ell_{t-i-1}(\cT)} \\
		&= (1- \mu)^2 \mu^{t-1} t \\
		&= \Pr{\he_{\cT^*}(u_0) = t}.
	\end{align*}
	Hence
	\[
		\he_{\cT_n}(v_0) \convdis \he_{\cT^*}(u_0).
	\]
	By Lemma~\ref{le:equ} it follows that
	\[
		(\cT_n, v_0) \convdis \cT^*
	\]
	in the space $\fmT^\bullet$ of pointed plane trees.
\end{proof}

\subsubsection{The type III regime}

\begin{proof}[Proof of Proposition~\ref{pro:t3}]
We need to show that the following three properties are equivalent.	
	\begin{enumerate}	
		\item  $(\cT_n, v_0) \convdis \cT^{*}$ in $\fmT^\bullet$.
		\item $\he_{\cT_n}(v_0) \convp 1$.
		\item The maximum degree  $\Delta(\cT_n)$ satisfies $\Delta(\cT_n) = n + o_p(n)$.
	\end{enumerate}
It is clear that the first claim implies the second, and that the second claim implies the third. If $\Delta(\cT_n) = n + o_p(n)$, then the vertex with largest degree is with high probability the root \cite[Equation (20.2)]{MR2908619}. So in this case, it follows that $\he_{\cT_n}(v_0) \convp 1$. Hence the third claim implies the second. By Lemma~\ref{le:equ}, it also holds that the second claim implies the first.
\end{proof}

\subsection{Truncated limits and large degrees}

\begin{proof}[Proof of Theorem~\ref{te:trunc}]
		Let $\cT_n^\bullet$ denote the tree $\cT_n$ pointed at the uniformly at random selected vertex $v_0$, and let $v_0, v_1, \ldots$ denote the path from $v_0$ to the root of $\cT_n$.
		Suppose that the weight sequence $\mathbf{w}$ has type II or III. Let $T^\bullet = (T,x)$ denote a finite plane tree that is pointed at vertex different from its root, and let $k$ denote the height of the pointed vertex in $T$. The inner root of the tree $T^\bullet$ will be denoted by $o$.
		
		For all $k_1, k_2 \ge 0$ let $\cE_{k_1, k_2}(T^\bullet)$ denote the set of pointed plane trees obtained by connecting the root vertices of $k_1$ arbitrary plane-trees from the left to the vertex $o$ of $T^\bullet$, and $k_2$ from the right. As we argued in Equation~\eqref{eq:gg2}, there is an ordering $\bar{d}_1, \ldots, \bar{d}_\ell$ of the outdegrees of the vertices $v \ne o$ of the tree $T$ such that with $\bar{d}_0 = d^+_T(o)$ it holds that
		\begin{align}
			\label{eq:starting}
			\Pr{f^\bullet(\cT_n^\bullet, v_k) \in \cE_{k_1, k_2}(T^\bullet)} = \Pr{(Y_0, \ldots, Y_\ell) = (k_1 + k_2 + \bar{d}_0,\bar{d}_1,  \ldots, \bar{d}_\ell)}.
		\end{align}
		For each $n$, let
\[
\cE_{n} = \bigcup_{\substack{k_1 \ge 0, k_2 \ge 0 \\ k_1 + k_2 + \bar{d}_0 \ge \Omega_n}} \cE_{k_1, k_2}.
\]
Setting $\bar{d}_0 = d^+_{T^\bullet}(o)$, it follows that
\begin{align}
\label{eq:s1}
\Pr{H(\cT_n,v_0,\Omega_n) \in \cE_{n}(T^\bullet)} &= \sum_{r \ge \Omega_n}  \sum_{\substack{k_1 + k_2 + \bar{d}_0 = r}} \Pr{(Y_0, \ldots, Y_\ell) = (r,\bar{d}_1,  \ldots, \bar{d}_\ell)} \nonumber \\
&= \sum_{r \ge \Omega_n} (r-\bar{d}_0+1)\Pr{(Y_0, \ldots, Y_\ell) = (r,\bar{d}_1,  \ldots, \bar{d}_\ell)}
\end{align}
By Equation~\eqref{eq:schuessel2}, it follows that
\begin{multline*}
\Pr{H(\cT_n,v_0,\Omega_n) \in \cE_{n}(T^\bullet)} = \\ \Exb{\frac{\sum_{r \ge \Omega_n} rN_r}{n} \prod_{i=1}^\ell \frac{N_{\bar{d}_i}}{n} + (1-\bar{d}_0) \frac{\sum_{r \ge \Omega_n} N_r}{n} \prod_{i=1}^\ell \frac{N_{\bar{d}_i}}{n} + O\left( \frac{\bar{d}_0 \sum_{r \ge \Omega_n} rN_r}{n^2}\right)}.
\end{multline*}
By Equation~\eqref{eq:om} we know that
\[
\frac{\sum_{r \ge \Omega_n} rN_r}{n} = 1 - \mu + o_p(1).
\]
Janson's result \cite[Thm. 11.4]{MR2908619} implies that for each fixed $j$
\[
	\frac{N_j}{n} \convp \Pr{\xi = j}
\]
and hence for each fixed $K$
\[
\frac{\sum_{r \ge K} N_r}{n} = 1 - \frac{\sum_{r < K} N_r}{n} \convp \Pr{\xi \ge K}.
\]
Consequently,
\[
\frac{\sum_{r \ge \Omega_n} N_r}{n} = o_p(1).
\]
By dominated convergence, it follows that
\[
\Pr{H(\cT_n,v_0,\Omega_n) \in \cE_{n}(T^\bullet)} \to (1 - \mu) \prod_{i = 1}^\ell \Pr{\xi = \bar{d}_i}.
\]
Let $V \subset \VHT^\bullet$ denote the subset corresponding to the vertices of the tree $T^\bullet$. Recall that the pointed vertex in $T^\bullet$ has height $k$. It holds that
\[
\Pr{d^+_{\bar{\cT}^{*}}(v) = d^+_{T^\bullet}(v) \text{ for all $v \in V \setminus \{ u_k\}$}, d^+_{\bar{\cT}^{*}}(u_k) = (\infty, \infty) } = (1-\mu)\prod_{i=1}^\ell \Pr{\xi = \bar{d}_i}.
\]
It readily follows that
\[
	H(\cT_n,v_0,\Omega_n) \convdis \bar{\cT}^*
\]
in the space $\fmT^\bullet$.
\end{proof}

Before proceeding with the proof of the main results, we make the following observation.

\begin{lemma}
	\label{le:bound}
	It holds that \[\tilde{D}_n \le (1 - \nu + o(1))n\] with probability tending to $1$ as $n$ becomes large. 
\end{lemma}
\begin{proof}
	By Equation~\eqref{eq:dn} we have that
	\[
		\tilde{D}_n \eqdist (d^+_{\cT_n}(o) \mid d^+_{\cT_n}(o) > \Omega_n) \le \Delta(\cT_n)
	\]
	and \cite[Lem. 19.32]{MR2908619} states that
	\begin{align}
		\sum_{v \in \cT_n} d^+_{\cT_n}(v) \one_{d^+_{\cT_n}(v)} = (1 - \nu + o_p(1))n.
	\end{align}
	Hence $\Delta(\cT_n) \le (1 - \nu + o_p(1))n$.
\end{proof}

\begin{proof}[Proof of Theorem~\ref{te:approx}]	
Let $\epsilon>0$ be given, and $m \ge 1$ be arbitrarily large but fixed. The height of the pointed vertex in $\bar{\cT}_n^*$ is stochastically bounded. Hence if $M_1 \ge 1$ is large  enough, the probability for this height to be larger than $M_1$ is less than $\epsilon$ for all $n$. 

The total size of the tree obtained by pruning  $\bar{\cT}_n^*$ at its  vertex with large degree, such that at most $m$ trees to the left and right of its spine offspring remain, is also stochastically bounded. Hence if $M_2 \ge 1$ is large enough the probability for this size to be larger than $M_2$ is at most $\epsilon$ for all $n$. 

By Lemma~\ref{le:bound}, we know that there is a sequence $t_n = o(1)$ such that the probability for the root-degree of $\bar{\cT}_n^*$ to be larger than $(1 - \nu + t_n)n$ tends to zero as $n$ becomes large. By modifying $t_n$ for finitely many $n$ we may also assume that additionally this probability is less than $\epsilon$ for all $n$.

Let $x_1, \ldots, x_r \in \VHT^\bullet$ be given vertices, and let $M_3$ denote the distance from the center $u_0$ to the youngest common ancestor of $x_1, \ldots, x_r$. 

Let $M > M_1, M_2, M_3$ be a fixed constant. Let $V \subset \VHT^\bullet$ correspond to the vertex set of a pointed tree $(T,x)$ with at most $M$ vertices such that $1 \le \he_T(x) \le M$ and the root $o$ has at most $m$ offspring vertices to the left and to the right of its unique son that lies on the spine. 

We are going to show that
\begin{align}
	\label{eq:toshow}
	\sup_{\substack{(k_1, k_2) \\ \Omega_n \le k_1 + k_2 \le (1 - \nu + t_n)n}} \left| \frac{\Pr{d^+_{\cT_n^\bullet}(v) = d^+_T(v) \text{ for $v \in V \setminus\{o\}$}, d^+_{\cT_n^\bullet}(o) = (k_1, k_2)  } }{\Pr{d^+_{\bar{\cT}_n^*}(v) = d^+_T(v) \text{ for $v \in V \setminus\{o\}$}, d^+_{\bar{\cT}_n^*}(o) = (k_1, k_2)  } } -1 \right| \to 0.
\end{align}
Note that the nominator and denominator are either both non-zero or both zero, and we will tacitly only consider the case where this expression is well-defined. In particular, this entails considering only trees $T$ such that
\[
\Pr{\xi = d_T^+(v)} >0
\]
for all vertex $v \in V(T) \setminus \{o\}$. There are only finitely many choices for $T$ and $V$. Hence the limit \eqref{eq:toshow}  implies that
\[
	d_{\textsc{TV}}( (d^+_{H(\cT_n,v_0, \Omega_n)}(x_i))_{1 \le i \le r}, (d^+_{\bar{\cT}_n^*}(x_i))_{1 \le i \le r}) \le 3 \epsilon
\]
for large enough $n$. As $\epsilon>0$ was arbitrary, this suffices to prove the claim.

It remains to verify the limit in \eqref{eq:toshow}. We may assume that $n$ is large enough such that $\Omega_n > M$. By Equation~\eqref{eq:starting}, we may order the outdegrees of $T$ by $(\bar{d}_0, \ldots, \bar{d}_\ell)$ such that $\bar{d}_0 = d^+_T(o)$ and for all $k_1, k_2$
\begin{multline*}
\Pr{d^+_{\cT_n^\bullet}(v) = d^+_T(v) \text{ for $v \in V \setminus\{o\}$}, d^+_{\cT_n^\bullet}(o) = (k_1, k_2)}  \\
= \Pr{(Y_0, \ldots, Y_\ell) = (K,\bar{d}_1,  \ldots, \bar{d}_\ell)}.
\end{multline*}
with \[
K = k_1 + k_2 + \bar{d}_0.
\] 
By the construction of $\bar{\cT}_n^*$ it holds that
\begin{multline*}
\Pr{d^+_{\bar{\cT}_n^*}(v) = d^+_T(v) \text{ for $v \in V \setminus\{o\}$}, d^+_{\bar{\cT}_n^*}(o) = (k_1, k_2)  } = \\K^{-1} (1-\mu) \Pr{\tilde{D}_n = K} \prod_{i=1}^\ell \Pr{\xi = \ell}.
\end{multline*}
Thus
\begin{multline*}
\frac{\Pr{d^+_{\cT_n^\bullet}(v) = d^+_T(v) \text{ for $v \in V \setminus\{o\}$}, d^+_{\cT_n^\bullet}(o) = (k_1, k_2)  } }{\Pr{d^+_{\bar{\cT}_n^*}(v) = d^+_T(v) \text{ for $v \in V \setminus\{o\}$}, d^+_{\bar{\cT}_n^*}(o) = (k_1, k_2)  } } = \\
\frac{K \Pr{Y_0 = K}}{(1-\mu)\Pr{\tilde{D}_n = K}} \Pr{(Y_1, \ldots, Y_\ell) = (\bar{d}_1,  \ldots, \bar{d}_\ell) \mid Y_0 = K} \left( \prod_{i=1}^\ell \Pr{\xi = \ell} \right)^{-1}.
\end{multline*}
By Equation~\eqref{eq:dn} and $K \ge \Omega_n$ it holds that
\[
\Pr{\tilde{D}_n = K} = \Pr{d^+_{\cT_n}(o) = K} / \Pr{d^+_{\cT_n}(o) > \Omega_n}.
\]
For any integer $k \ge 0$ it holds by the discussion in Section~\eqref{sec:preli} and in particular Equation~\eqref{eq:shift} that
\[
	\Pr{d_{\cT_n}^+(o) =k} = \frac{nk}{n-1} \Pr{Y_0=k}.
\]
See also \cite[Lemma 15.7]{MR2908619}. Hence 
\[
\Pr{\tilde{D}_n = K} = \frac{K \Pr{Y_0 = K}}{ \sum_{k >\Omega_n} k \Pr{Y_0=k}}.
\]
It follows by Equation~\eqref{eq:om} that the term
\[
\frac{K \Pr{Y_0 = K}}{(1-\mu)\Pr{\tilde{D}_n = K}} = \frac{\sum_{k >\Omega_n} k \Pr{Y_0=k}}{1-\mu} 
\]
does not depend on $K$ at all and converges toward $1$. Thus, in order to verify the limit \eqref{eq:toshow}, it remains to show that
\begin{align}
\label{eq:laststand}
\Pr{(Y_1, \ldots, Y_\ell) = (\bar{d}_1,  \ldots, \bar{d}_\ell) \mid Y_0 = K} \left( \prod_{i=1}^\ell \Pr{\xi = \ell} \right)^{-1} \to 1
\end{align}
uniformly for all $\Omega_n \le K \le (1 - \mu + t_n)n$. Note that
\begin{multline*}
\Pr{(Y_1, \ldots, Y_\ell) = (\bar{d}_1,  \ldots, \bar{d}_\ell) \mid Y_0 = K} = \\ \Pr{(Y_1^{(n-1-K,n-1)}, \ldots, Y_\ell^{(n-1-K,n-1)}) = (\bar{d}_1,  \ldots, \bar{d}_\ell) }
\end{multline*}
For ease of notation, let us set $Y_i' = Y_i^{(n-1-K,n-1)}$ for all $i$ and let $N_k^{(n-1-K,n-1)} = N_k'$ denote the number of indices with $i$ with $Y_i' = k$. Similarly as in   Equation~\eqref{eq:s0} it holds that
\begin{align*}
\Pr{(Y_1', \ldots, Y_\ell') = (\bar{d}_1,  \ldots, \bar{d}_\ell) \mid N_0', N_1', \ldots} &= \prod_{i=1}^\ell \frac{N'_{\bar{d}_i} - c_i }{n-i} 
\end{align*}
with $c_i$ denoting the number of $1 \le j <i$ with $\bar{d}_j = \bar{d}_i$. It is elementary that
\[
\prod_{i=1}^\ell \frac{N'_{\bar{d}_i} - c_i }{n-i} 
= (1 + O(n^{-1})) \prod_{i=1}^\ell \frac{N'_{\bar{d}_i}}{n}.
\]
with the implicit bound in the $O(n^{-1})$ term not depending on $K$. 

 Recall that in Section~\ref{sec:types} we defined $\phi(z) = \sum_{k \ge 0} \omega_k z^k$,  $\psi(z) = z \phi'(z) / \phi(z)$, and a parameter $\tau$.  As the weight sequence $(\omega_i)_i$ has type II or III, it holds that $\tau=\rho_\phi$ is the radius of convergence of $\phi(z)$. In Section~\ref{sec:gwt} we defined furthermore $\Pr{\xi=k} = \omega_k \tau^k/ \phi(\tau)$ for all $k$. Janson~\cite[Theorem 11.6]{MR2908619} gave the following result. The function
\[
	\tau: [0, \infty[ \to [0, \infty], \quad x \mapsto \sup\{t \le \rho \mid \psi(t) \le x\}.
\]
is continuous. For $x \le \nu = \mu$ it holds that $\tau(x)$ is the unique number with $\psi(\tau(x))=x$, and for $x > \nu$ it holds that $\tau(x) = \rho_\phi = \tau$. Furthermore, for each fixed non-negative integer $d$ it holds uniformly for all $m\le n$ that
\[
\frac{N_d^{(m,n)}}{n} - \frac{\omega_d \tau(m/n)^d}{\phi(m/n)} \convp 0.
\]

We assumed that $K \le (1 - \mu +t_n)n$ with $t_n = o(1)$. In particular, 
\[
(n-1-K)/(n-1) \sim \mu
\]
uniformly for all $K$. Thus 
\[
\tau((n-1-K)/(n-1)) \sim \tau
\]
and consequently
\[
	\frac{N_{\bar{d}_i}'}{n} - \Pr{\xi = \bar{d}_i} \convp 0
\]
uniformly for all $K$. As $\Pr{\xi = \bar{d}_i}>0$ for all $i$, it follows by dominated convergence that 
\[
\Pr{(Y_1', \ldots, Y_\ell') = (\bar{d}_1,  \ldots, \bar{d}_\ell) }\left( \prod_{i=1}^\ell \Pr{\xi = \ell} \right)^{-1} \to 1 
\]
uniformly for all $K$. This verifies Equation~\eqref{eq:laststand} and hence completes  the proof.
\end{proof}

\section*{Acknowledgement}
I warmly thank the editor and referees for the helpful suggestions and thorough reading of the manuscript.

\bibliographystyle{siam}
\bibliography{pointed}

\end{document}